\documentclass[11pt,a4paper]{amsart}
\newcommand{\Title}{A Complete Characterization of Regular  Inclusions of Finite Dimensional $C^*$-algebras}%
\newcommand{\ShortTitle}{\Title}%

\newcommand{\AuthorOne}{Keshab Chandra Bakshi}%
\newcommand{\AuthorOneAddr}{%
	Department of Mathematics and Statistics, Indian Institute of Technology Kanpur, Uttar Pradesh 208 016, India
}%

\newcommand{\AuthorOneEmail}{%
	 bakshi209@gmail.com, keshab@iitk.ac.in
}%

\newcommand{\AuthorTwo}{Indrajit Ghosh}%
\newcommand{\AuthorTwoAddr}{%
	Department of Mathematics and Statistics, Indian Institute of Technology Kanpur, Uttar Pradesh 208 016, India
}%

\newcommand{\AuthorTwoEmail}{%
	indrajitghosh912@gmail.com, indrajitg@iitk.ac.in
}%

 \newcommand{\AuthorThree}{Sumit Kumar}%
\newcommand{\AuthorThreeAddr}{Department of Mathematics and Statistics, Indian Institute of Technology Kanpur, Uttar Pradesh 208 016, India
}%
\newcommand{\AuthorThreeEmail}{sumitkumar.sk809@gmail.com}

\newcommand{\SubjectClassText}{Primary 46L05, 47L40; Secondary 46L10}

\newcommand{\Keywords}{Regularity, $C^*$-algebra, Normalisers, Finite Dimensional $C^*$-algebras}

\newcommand{\pdfTitle}{\Title}
\newcommand{\pdfAuthor}{Indrajit Ghosh}
\newcommand{\pdfSubject}{Mathematics, Research paper}
\newcommand{\pdfKeywords}{\Keywords}
\newcommand{\pdfCreator}{TeXlive}
\newcommand{\pdfCreationDate}{\today}
\newcommand{\pdfColorLink}{true}
\newcommand{\pdfLinkColor}{cyan}
\newcommand{\pdfUrlColor}{blue}
\newcommand{\pdfCiteColor}{magenta}

\usepackage[top=0.9in, bottom=1in, left=0.7in, right=0.7in]{geometry}
\usepackage{amsmath, amssymb, amsthm} 
\usepackage[utf8]{inputenc}
\usepackage[T1]{fontenc}
\usepackage{mathtools}
\usepackage{mathrsfs} 
\usepackage{xfrac} 
\usepackage{dsfont} 
\usepackage{array}
\usepackage{verbatim}
\usepackage{graphicx}
\usepackage{mdframed}
\usepackage{enumitem} 
\usepackage{hyperref}
\hypersetup{
	pdftitle={\pdfTitle},
	pdfauthor={\pdfAuthor},
	pdfsubject={\pdfSubject},
	pdfcreationdate={\pdfCreationDate},
	pdfcreator={\pdfCreator},
	pdfkeywords={\pdfKeywords},
	colorlinks=\pdfColorLink,
	linkcolor={\pdfLinkColor},
	urlcolor=\pdfUrlColor,
	citecolor=\pdfCiteColor,
	pdfpagemode=UseOutlines,
}
\usepackage{tikz-cd} 
\usepackage{lipsum}
\usetikzlibrary{matrix,arrows}
\usepackage[english]{babel}
\usepackage{lmodern}
\usepackage{bbm} 
\usepackage[dvipsnames]{xcolor}
\usepackage[most]{tcolorbox}
\usepackage{xparse}

\theoremstyle{plain}
\newtheorem{theorem}{Theorem}[section]
\newtheorem{prop}[theorem]{Proposition}
\newtheorem{lem}[theorem]{Lemma}
\newtheorem{cor}[theorem]{Corollary}

\theoremstyle{definition}
\newtheorem{definition}[theorem]{Definition}
\newtheorem{example}[theorem]{Example}

\theoremstyle{remark}
\newtheorem{remark}[theorem]{Remark}

\numberwithin{equation}{section}

\newtheoremstyle{ser}
{8pt}
{8pt}
{\it}
{}
{\sf}
{:}
{6mm}
{}

\theoremstyle{ser}

\newtheoremstyle{serr}
{8pt}
{8pt}
{\normalfont}
{}
{\sf}
{.}
{6mm}
{}

\theoremstyle{serr}

\theoremstyle{ser}

\theoremstyle{ser}
\newtheorem{qn}{Question}

\newtheoremstyle{collabquestion}
  {8pt}
  {8pt}
  {\normalfont}
  {}
  {\sffamily\bfseries\color{blue!70!black}}
  {.}
  {.5em}
  {}

\theoremstyle{collabquestion}
\newtheorem{qninner}{Question}

\definecolor{indraRed}{rgb}{0.593, 0.183, 0.183}
\definecolor{indraPink}{rgb}{0.858, 0.188, 0.478}
\definecolor{indraBlue}{rgb}{0, 0.199, 0.398}
\definecolor{madridBlue}{rgb}{0.199, 0.199, 0.695}
\definecolor{metropolisThemeColor}{rgb}{0.105, 0.214, 0.234}
\definecolor{metropolisBarColor}{rgb}{0.984, 0.0.515, 0.015}
\definecolor{UBCblue}{rgb}{0.04706, 0.13725, 0.26667} 
\definecolor{UBCgrey}{rgb}{0.3686, 0.5255, 0.6235} 

\makeatletter
\def\mathcolor#1#{\@mathcolor{#1}}
\def\@mathcolor#1#2#3{%
	\protect\leavevmode
	\begingroup
	\color#1{#2}#3%
	\endgroup
}
\makeatother

\makeatletter
\def\ps@headings{\ps@empty
  \def\@evenhead{\normalfont\scriptsize\hfil \leftmark{}{}\hfil}%
  \def\@oddhead{\normalfont\scriptsize\hfil \rightmark{}{}\hfil}%
  \let\@mkboth\markboth
  \def\@evenfoot{\normalfont\scriptsize\hfil\thepage\hfil}%
  \def\@oddfoot{\normalfont\scriptsize\hfil\thepage\hfil}%
}
\makeatother

\makeatletter
\def\author@andify{%
  \nxandlist {\unskip ,\penalty-1 \space\ignorespaces}%
    {\unskip {} \@@and~}%
    {\unskip \penalty-2 \space \@@and~}%
}
\makeatother

\definecolor{indraRed}{rgb}{0.593, 0.183, 0.183}
\definecolor{indraPink}{rgb}{0.858, 0.188, 0.478}
\definecolor{indraBlue}{rgb}{0, 0.199, 0.398}
\definecolor{madridBlue}{rgb}{0.199, 0.199, 0.695}
\definecolor{metropolisThemeColor}{rgb}{0.105, 0.214, 0.234}
\definecolor{metropolisBarColor}{rgb}{0.984, 0.0.515, 0.015}
\definecolor{UBCblue}{rgb}{0.04706, 0.13725, 0.26667} 
\definecolor{UBCgrey}{rgb}{0.3686, 0.5255, 0.6235} 

\newcommand{\spn}{{\operatorname{span}\,}}

\newcommand{\C}{\mathbb{C}}

\newcommand{\N}{\mathbb{N}}

\NewDocumentCommand{\mn}{m O{\mathbb{C}}}
{\mathbb{M}_{#1}(#2)} 

     \newcommand{\sA}{\mathcal A}		
     \newcommand{\sB}{\mathcal B}		
     \newcommand{\sC}{\mathcal C}		
     \newcommand{\sD}{\mathcal D}

     \newcommand{\sG}{\mathcal G}		
     \newcommand{\sH}{\mathcal H}

     \newcommand{\sN}{\mathcal N}

     \newcommand{\sU}{\mathcal U}

     \newcommand{\sZ}{\mathcal Z}

\NewDocumentCommand{\nor}{g}
  {\sN\IfValueT{#1}{_{#1}}}

\NewDocumentCommand{\unor}{g}
  {\sU\sN\IfValueT{#1}{_{#1}}}

\NewDocumentCommand{\grnor}{g}
  {\sG\sN\IfValueT{#1}{_{#1}}}

\begin{document}
	
        \title[\ShortTitle]{\MakeUppercase\Title}
	\author{\AuthorOne}
	\address[Bakshi]{\AuthorOneAddr}
	
	\email{\AuthorOneEmail}

    \author{\AuthorTwo}
	\address[Ghosh]{\AuthorTwoAddr}
	\email{\AuthorTwoEmail}

    \author{\AuthorThree}
	\address[Kumar]{\AuthorThreeAddr}
	\email{\AuthorThreeEmail}
	
	\date{}
	\subjclass{\SubjectClassText}
	
	
	\begin{abstract}
		
We give a complete characterization of regular (in the sense of Kumjian and Renault) unital inclusions of finite-dimensional $C^*$-algebras. For subalgebras $\bigoplus_j( \mathbb{M}_{d_j}(\mathbb{C}) \otimes \mathbb{I}_{p_j})$ of $\mathbb{M}_n(\mathbb{C})$, we show that regularity depends only on equality of the multiplicities $p_j$, while unitary regularity---characterized recently by the first author and Silambarasan---additionally requires equality of the $d_j$; we recover the latter via a streamlined alternative proof. Extending this to inclusions of arbitrary finite-dimensional $C^*$-algebras, encoded by an inclusion matrix $\Lambda$, we show that regularity is equivalent to an explicit row/column condition on $\Lambda$---coinciding with the normalizer matrix introduced by the first author and Silambarasan ---so that the device used there to detect unitary regularity is shown to characterize regularity in general; unitary regularity is recovered by a further dimension-equality condition. 
	\end{abstract}

	\maketitle%
        \thispagestyle{empty}%

    \tableofcontents
	


    \section{Introduction}
    \noindent In his seminal paper \cite{Dix1954}, Dixmier initiated the study of unitary normalizers of von Neumann subalgebras of $\mathrm{II}_1$-factors to investigate maximal abelian self-adjoint subalgebras (masas). Building on Dixmier's work, Feldman and Moore \cite{Feldman_Moore} further developed the theory by establishing a correspondence between Cartan subalgebras and measured equivalence relations. Inspired by these developments, Kumjian \cite{Kumjian_86} and Renault \cite{Renaultbook} independently introduced the notion of regularity in the $C^*$-algebraic setting, thereby laying the foundation for the theory of Cartan subalgebras in $C^*$-algebras. Later, Exel \cite{Exel2011} introduced the notion of noncommutative Cartan subalgebras in $C^*$-algebras and developed the theory of regularity for general noncommutative inclusions.

Very recently, the first author and Silambarasan \cite{Bakshi2026} studied unitary regularity in finite-dimensional $C^*$-algebras and provided a characterization of unitarily regular inclusions in this setting. However, the characterization of general regular subalgebras of finite-dimensional $C^*$-algebras was left open. Although the regularity of inclusions has been extensively studied across various classes of operator algebras, a comprehensive understanding within the finite-dimensional context remains missing from the literature. Motivated by this gap, our aim is to develop a systematic framework for studying regular inclusions in finite-dimensional $C^*$-algebras. In this article, we provide a complete characterization of general regularity for unital inclusions of finite-dimensional $C^*$-algebras, alongside a more streamlined alternative characterization of unitary regularity (cf. \cite{Bakshi2026}). We emphasize that regularity is a property of the specific embedding rather than the abstract isomorphism class of the subalgebra. Indeed, two distinct embeddings of an abstract $C^*$-algebra $\sB$ into $\sA$ can yield entirely different regularity outcomes (see Remark \ref{rem:abs_vs_concrete}). 
Fortunately, regularity is invariant under automorphisms of the ambient algebra: if $\sB \subseteq \sA$ is a concrete subalgebra and $\sigma \in \operatorname{Aut}(\sA)$, then $\sB$ is regular in $\sA$ if and only if $\sigma(\sB)$ is regular in $\sA$ (see Proposition \ref{prop:autmorph_pres_nor}). With this framework established, we now discuss our main results.

First, we characterize the (unitarily) regular unital $C^*$-subalgebras of $\mn{n}$ (see \S~\ref{sec:char_mn}). By the Artin--Wedderburn theorem (Theorem \ref{thm:artin-wedderburn}), any finite-dimensional $C^*$-algebra is isomorphic to a direct sum $\bigoplus_{j=1}^k \mn{d_j}$ for some integers $d_1, \dots, d_k \in \N$. Embedding this algebra into $\mn{n}$ assigns a non-zero multiplicity to each block. Consequently, up to unitary conjugation, any unital $C^*$-subalgebra of $\mn{n}$ takes the standard form (see Proposition \ref{Spatial-inclusion}):
$$
\bigoplus_{j=1}^k \left( \mn{d_j} \otimes \mathbb{I}_{p_j} \right),
$$
for some $p_1, \dots, p_k\in \N$ satisfying $\sum_{j=1}^k p_jd_j = n$. The vectors $\vec{p} := (p_1, \dots, p_k)$ and $\vec{d}:=(d_1, \dots, d_k)$ are called the \emph{multiplicity vector} and \emph{dimension vector} of the subalgebra, respectively.

Next, we generalize this to an arbitrary finite-dimensional $C^*$-algebra. For a vector $\vec{n}:=(n_1, \dots, n_l)\in \N^l$, we denote $\mn{\vec{n}}:= \bigoplus_{i = 1}^l \mn{n_i}$. When $\vec{n}$ is a scalar $n$, we saw that unital $C^*$-subalgebras correspond to vectors $\vec{p}, \vec{d}\in \N^k$ such that $\vec{p}\cdot \vec{d}=n$. For an arbitrary $\vec{n}\in \N^l$, this structure naturally generalizes to an $l\times k$ matrix $\Lambda$, called the \emph{inclusion matrix}, and a dimension vector $\vec{d}\in \N^k$ satisfying 
$$
\Lambda \vec{d}= \vec{n}.
$$
Together, $\Lambda$ and $\vec{d}$ determine a concrete unital $C^*$-subalgebra of $\mn{\vec{n}}$. Thus, throughout this article, the phrase ``$\mn{\vec{d}}$ is a subalgebra of $\mn{\vec{n}}$ with inclusion matrix $\Lambda$'' specifically refers to the image of the embedding of $\mn{\vec{d}}$ into $\mn{\vec{n}}$ induced by $\Lambda$ (see \S~\ref{sec:char_gen} for details).
The main question we address in this article is:

\begin{qn}
    What properties of the inclusion matrix $\Lambda_{l \times k}$ characterize whether the corresponding subalgebra $\mn{\vec{d}}$ is (unitarily) regular in $\mn{\vec{n}}$, and conversely?
\end{qn}

To answer this in the case of $\mn{n}$, we prove the following theorem, which characterizes regular subalgebras of $\mn{n}$. Specifically, we show that in $\mn{n}$, regularity is entirely determined by the multiplicities of the matrix blocks present in the subalgebra.\\

\noindent \textbf{Theorem A:} (Theorem \ref{thm:regularity-char-mnc})
\emph{The subalgebra $\bigoplus_{j=1}^{k} \left( \mn{d_j} \otimes \mathbb{I}_{p_j} \right)$ is regular in $\mn{n}$ if and only if
$$
p_1 = \cdots = p_k.
$$}\\
Regarding unitary regularity, Theorem \ref{thm:uni-reg-char-mnc} provides a complete characterization for subalgebras of $\mn{n}$ (cf. \cite{Bakshi2026}). More precisely,
\emph{the subalgebra $\bigoplus_{j=1}^{k} \left( \mn{d_j} \otimes \mathbb{I}_{p_j} \right)$ is unitarily regular in $\mn{n}$ if and only if
$$
p_1 = \cdots = p_k \text{ and } d_1 = \cdots = d_k.
$$}\\

Observe that the preceding results allow us to determine the total number of regular and unitarily regular subalgebras of $\mn{n}$ simply by counting the number of partitions of $n$ encoded by the corresponding vectors $\vec{p}$ and $\vec{d}$. The total number of pairs $(\vec{p}, \vec{d})$ satisfying $p_1 = \cdots = p_k$ is
\[
\sum_{S \mid n} 2^{S-1},
\]
while the total number of pairs satisfying both $p_1 = \cdots = p_k$ and $d_1 = \cdots = d_k$ is
\[
\sum_{d \mid n} \tau(d).
\]
These exact counting sequences are cataloged in the On-Line Encyclopedia of Integer Sequences (OEIS) as \href{https://oeis.org/search?q=A034729&go=Search}{\texttt{A034729}} and \href{https://oeis.org/search?q=A007425&go=Search}{\texttt{A007425}}, respectively. 
Their values for small $n$ are summarized in the following table (see table (\ref{tab:formula_values}) below):

\begin{table}[htbp]
\caption{Total number of regular and unitarily regular subalgebras of $\mn{n}$ for $1 \leq n \leq 10$}
\label{tab:formula_values}
\renewcommand{\arraystretch}{1.3}
\begin{tabular}{ccc}
\hline\hline
$n$ & $\displaystyle \sum_{S\vert n} 2^{S-1}$  & $\displaystyle \sum_{d\vert n} \tau(d)$ \\[3ex]
\hline
1  & 1   & 1 \\
2  & 3   & 3 \\
3  & 5   & 3 \\
4  & 11  & 6 \\
5  & 17  & 3 \\
6  & 39  & 9 \\
7  & 65  & 3 \\
8  & 139 & 10 \\
9  & 261 & 6 \\
10 & 531 & 9 \\
\hline\hline
\end{tabular}
\end{table}

Motivated by the ideas behind the proofs of these theorems, we extend our techniques to obtain a complete characterization of (unitarily) regular inclusions in general finite-dimensional $C^*$-algebras. To state this result, we require the following definition:

\begin{definition}[Regular Matrix]\label{def:regular_mat}
Let $\Lambda=(\Lambda_{ij})$ be an $l\times k$ matrix with nonnegative integer entries. For each $i=1,\ldots,l$, define the support of the $i$th row by
\[
S_i:=\{\,j:\Lambda_{ij}>0\,\}.
\]
The matrix $\Lambda$ is called \emph{regular} if the following conditions hold:
\begin{enumerate}
    \item[(i)] For each $i=1,\ldots,l$, there exists an integer $p_i\in\mathbb{N}$ such that
    \[
    \Lambda_{ij}=p_i,\qquad \forall\, j\in S_i.
    \]
    Equivalently, all positive entries in the $i$th row are equal.

    \item[(ii)] For each $i=1,\ldots,l$, any two columns indexed by $S_i$ are identical; that is,
    \[
    \vec{\Lambda}_j=\vec{\Lambda}_{j'},\qquad \forall\, j,j'\in S_i.
    \]
\end{enumerate}
\end{definition}

Here, we point out that the \emph{normalizer matrix} introduced in \cite[Definition 3.1]{Bakshi2026} is precisely the same as the regular matrix defined above (see Proposition \ref{prop:reg_mat-iff-normaliser_mat}). We adopt the term \emph{regular matrix} here to better align with our broader classification of general regular inclusions. With this definition in place, we are now ready to state our main result. 
\medskip

\noindent \textbf{Theorem B:} (Theorem \ref{thm:char_regularity})
\emph{Let $\vec{n} = (n_1,\dots,n_l)^T\in\mathbb{N}^l$, $\vec{d}=(d_1,\dots,d_k)^T\in\mathbb{N}^k$, and let $\Lambda=[\Lambda_{ij}]_{l\times k}$ satisfy
$\Lambda\vec{d}=\vec{n}.$ Then the algebra $\mn{\vec{d}}$ is regular in $\mn{\vec{n}}$ if and only if $\Lambda$ is regular.
}
\medskip

As our framework provides a unified approach to both general and unitary regularity, we also recover the characterization of unitarily regular inclusions established in \cite{Bakshi2026}.
\medskip

\noindent \textbf{Theorem C:} (Theorem \ref{thm:char_unitary_regularity})
\emph{Let $\vec{n} = (n_1,\dots,n_l)^T\in\mathbb{N}^l$, $\vec{d}=(d_1,\dots,d_k)^T\in\mathbb{N}^k$, and let $\Lambda=[\Lambda_{ij}]_{l\times k}$ satisfy
$\Lambda\vec{d}=\vec{n}.$
Then the algebra $\mn{\vec{d}}$ is unitarily regular in $\mn{\vec{n}}$ if and only if the following conditions hold:
\begin{enumerate}
    \item[(i)] $\Lambda$ is regular.
    \item[(ii)] For each $ i = 1, \dots, l$, we have $d_j = d_{j'}$ whenever $j, j' \in S_i$, 
\end{enumerate}
where
$$
S_i:=\{\,j:\Lambda_{ij}>0\,\}
$$
denotes the support of the $i$th row of $\Lambda$.}

We conclude this introduction by briefly outlining the structure of the article. Almost all results established here concern finite-dimensional $C^*$-algebras. In \S~\ref{sec:prelims}, we set up our notation and state some standard basic results without proof, which can be found in texts such as \cite{davidson_book}. In \S~\ref{sec:basic_definitions}, we define three types of regularity: unitary regularity, groupoid regularity, and general regularity. We demonstrate that regularity and groupoid regularity are equivalent (Theorem \ref{thm:reg-iff-gr-reg}). In Example \ref{ex:grop-reg-counter}-(ii), we provide an instance of a subalgebra that is regular but not unitarily regular. However, we show that for abelian subalgebras, all three types of regularity coincide. In \S~\ref{sec:char_mn}, we generalize this coincidence to homogeneous subalgebras of $\mn{n}$ (Corollary \ref{cor:hom-imply-reg-coincide}). Additionally, \S~\ref{sec:char_mn} addresses subalgebras of $\mn{n}$ directly. We establish several crucial results (e.g., Propositions \ref{prop:equal-mult-dim-imply-uni-reg}, \ref{prop:reg-impy-same-dim}, \ref{prop:equal-multi-imply-reg}) that build toward the characterizations of regularity (Theorem \ref{thm:regularity-char-mnc}) and unitary regularity (Theorem \ref{thm:uni-reg-char-mnc}) in $\mn{n}$. A consequence of these characterizations is that the only regular maximal abelian self-adjoint subalgebra (masa) of $\mn{n}$ is, up to unitary conjugation, the standard diagonal masa (Corollary \ref{cor:reg_masa_mn}). Finally, in \S~\ref{sec:char_gen}, we generalize the characterizations from the preceding section to arbitrary unital inclusions of finite-dimensional $C^*$-algebras $\sB \subseteq \sA$. We introduce essential notation for general concrete subalgebras of $\mn{\vec{n}}$ and provide the complete characterization of regular inclusions in terms of their inclusion matrices (Theorem \ref{thm:char_regularity}), followed by the corresponding characterization for unitary regularity (Theorem \ref{thm:char_unitary_regularity}).
    
    \section{Preliminaries on Finite Dimensions}
    \label{sec:prelims}
    
The primary purpose of this section is to establish the notation and conventions used throughout the article. We also recall several standard results on finite-dimensional representations of $C^*$-algebras. Although these results can be found in standard references such as \cite{davidson_book}, we present them here in the form most convenient for our subsequent development.

\noindent\textbf{Notations.} Throughout this article, given $d, p \in\N$ we will use the notation $\mn{d} \otimes \mathbb{I}_p$ to denote the unital $C^*$-subalgebra of $\mn{dp}$ defined by
\[
\mn{d} \otimes \mathbb{I}_p
:=
\left\{
A\otimes \mathbb{I}_p
:=
\begin{pmatrix}
A & & & \\
& A & & \\
& & \ddots & \\
& & & A
\end{pmatrix}
:\;
A\in \mn{d}
\right\}
\subseteq \mn{dp},
\]
where the matrix has $p$ diagonal blocks, each equal to $A$.

We shall also use the following notation for block diagonal subalgebras. Let
\[
\mn{d_j} \otimes \mathbb{I}_{p_j}\subseteq \mn{d_j p_j},
\qquad j = 1, \dots, k.
\]
Then
\[
\bigoplus_{j = 1}^{k} \left( \mn{d_j} \otimes \mathbb{I}_{p_j} \right)
:=
\left\{
\begin{pmatrix}
A_1 & & & \\
& A_2 & & \\
& & \ddots & \\
& & & A_k
\end{pmatrix}
:\;
A_j\in \mn{d_j} \otimes \mathbb{I}_{p_j}\subseteq \mn{d_j p_j},
\; 1 \le j \le k
\right\},
\]
which is a unital $C^*$-subalgebra of $\mn{n}$, where
\[
n=\sum_{j=1}^{k} d_j p_j.
\]

The following proposition is straightforward.

\begin{prop}\label{prop:abs-iso-concrete}
The map
\begin{align*} 
\Phi:\mn{d} & \to \mn{d} \otimes \mathbb{I}_p \\ A &\mapsto A \otimes \mathbb{I}_p. 
\end{align*}
is a unital $^*$-isomorphism. Consequently, the unital $C^*$-algebra
$\mn{d} \otimes \mathbb{I}_p \subseteq \mn{dp}$ is abstractly
$^*$-isomorphic to $\mn{d}$.
\end{prop}

\begin{theorem}[Artin-Wedderburn, \,{\cite[III.1.1]{davidson_book}}]\label{thm:artin-wedderburn}
    Let $\sB$ be a finite dimensional $C^*$-algebra. Then 
    \[
    \sB \cong \mn{d_1} \oplus \dots \oplus \mn{d_k}
    \]
    for some $d_1, \dots, d_k \in \N$.
\end{theorem}

\begin{lem}
    Let $\pi:\mn{d}\to \sB(\sH)$ be a unital faithful representation with $\dim (\sH) < \infty$. Then $\dim(\sH) = dp$ for some $p\in \N$ and there exist a unitary $u: \sH \to \C^d \otimes \C^p$ such that:
    \[
    u\pi(x)u^* = x \otimes \mathbb{I}_p.
    \]
\end{lem}

\begin{prop}[cf. {\cite[Corollary III.1.2]{davidson_book}}]\label{Spatial-inclusion}
    Let $\sB$ be a unital $C^*$-algebra and let $\pi:\sB \to \mn{n}$ be a unital $*$-homomorphism for some $n \in \N$. Then there exist $d_1, \dots, d_k\in \N\text{ and } p_1, \dots, p_k \in \N\cup \{0\}$ with
    \[
    \sum_{j=1}^k d_j p_j = n,
    \]
    and a unitary $u \in \mn{n}$ such that
    \[
    u\pi(a)u^*
    =
    \bigoplus_{j=1}^k \left(\mn{d_j} \otimes \mathbb{I}_{p_j}\right),
    \qquad \text{for all } a \in \sB.
    \]
    In particular, every unital $C^*$-subalgebra of $\mn{n}$ is, up to conjugation by a unitary in $\mn{n}$, of the above form where each $p_j$'s can be chosen to be non-zero.
    Moreover, the vectors \[ \vec{d} := (d_1,\dots,d_k), \qquad \vec{p} := (p_1,\dots,p_k), \] called the \emph{dimension vector} and the \emph{multiplicity vector}, respectively, are uniquely determined by $\sB$ (equivalently, by the homomorphism $\pi$) up to a simultaneous permutation of their entries.
\end{prop}

Below we record a well-known result about unital inclusion of finite dimensional $C^*$-algebras.
\begin{cor}
    Let $\sB$ be a unital (abstract) finite-dimensional $C^*$-algebra. Then there exists a unital embedding
    \[
    \pi:\sB \to \mn{n}
    \]
    for some $n \in \N$. Consequently, up to conjugation by a unitary $u \in \mn{n}$,
    \[
    u\pi(\sB)u^*
    =
    \bigoplus_{j=1}^k \left(\mn{d_j} \otimes \mathbb{I}_{p_j}\right),
    \]
    where $d_1, \dots, d_k, p_1, \dots, p_k \in \N$ satisfy
    \[
    \sum_{j=1}^k d_j p_j = n.
    \]
\end{cor}
\begin{proof}
    It follows directly from Artin-Wedderburn theorem and the preceeding proposition.
\end{proof}

\begin{prop}[{\cite[Corollary III.1.2]{davidson_book}}]\label{concrete-inclusion}
    Let $\pi:\sB= \mn{d_1}\oplus \dots \oplus \mn{d_k}\to \mn{n}$ be an unital $*$-homomorphism. Then there exist $p_1, \dots, p_k\in \N \cup \{0\}$ with $\sum_{j} d_j p_j = n$ and $u \in \sU(\mn{n})$ such that 
    \[
        u\pi(\sB)u^* = \bigoplus_{j= 1}^k \mn{d_j} \otimes \mathbb{I}_{p_j}.
    \] 
\end{prop}

\begin{prop}[\cite{davidson_book}]\label{prop:rep_of_fin_alg}
    The inequivalent irreducible representation of the finite dimensional algebra $\sB := \oplus_{j = 1}^{k} \mn{d_j}$ are precisely:
    \[
    \pi_j:\sB \to \mn{d_j} \text{ for } j = 1, \dots, k.
    \]
\end{prop}

Next, we recall the definition of a normalizer matrix from \cite{Bakshi2026} to establish its equivalence with our notion of a regular matrix.

\begin{definition}\cite[Definition 3.1]{Bakshi2026}
   A matrix $\Lambda=[\Lambda_{ij}]$ is called a \emph{normalizer matrix} if it satisfies the following conditions:
   \begin{enumerate}
       \item[(i)] In each row of $\Lambda$, all nonzero entries are equal.
       \item[(ii)] For any two rows $i$ and $k$, if there exists a column $j$ such that $\Lambda_{ij}=0$ and $\Lambda_{kj}\neq 0$, then for every column $l$,
       \[
       \Lambda_{il}\neq 0 \implies \Lambda_{kl}=0.
       \]
   \end{enumerate}
\end{definition}

\begin{prop}\label{prop:reg_mat-iff-normaliser_mat}
 A matrix $\Lambda$ is regular if and only if it is a normalizer matrix.
\end{prop}

\begin{proof}
\noindent ($\implies$) Suppose $\Lambda$ is a regular matrix. Condition (i) of a normalizer matrix holds trivially by definition. To verify condition (ii), suppose, for the sake of contradiction, that there exist rows $i$ and $k$, and a column $j$ such that $\Lambda_{ij}=0$ and $\Lambda_{kj}\neq 0$, yet the implication fails. This means there exists some column $t$ such that both $\Lambda_{it}\neq 0$ and $\Lambda_{kt}\neq 0$. Since $\Lambda_{kj}\neq 0$ and $\Lambda_{kt}\neq 0$, both indices $j$ and $t$ belong to the row support $S_k$. Because $\Lambda$ is a regular matrix, we must have identical columns $\vec{\Lambda}_{j}= \vec{\Lambda}_{t}$. However, this is a contradiction since $\Lambda_{ij}=0$ while $\Lambda_{it}\neq 0$. Thus, $\Lambda$ is a normalizer matrix.\\

\noindent ($\impliedby$) Conversely, suppose that $\Lambda$ is a normalizer matrix. It clearly satisfies condition (i) of a regular matrix. Suppose for a contradiction that it is not regular, meaning condition (ii) fails. Then there exists a row $i$ for which there exist distinct indices $j,j'\in S_i$ satisfying $\vec{\Lambda}_j\neq\vec{\Lambda}_{j'}$. This implies there is some row $k$ where the columns differ; without loss of generality, assume $\Lambda_{kj}\neq 0$ and $\Lambda_{kj'}=0$. Moreover, since $j,j'\in S_i$, we have $\Lambda_{ij}=\Lambda_{ij'}\neq 0$.   
Now, applying the second condition of the normalizer matrix definition to the pair of rows $k$ and $i$, and the column $j'$, we have $\Lambda_{kj'}=0$ and $\Lambda_{ij'}\neq 0$. This dictates that for every column $t$, if $\Lambda_{it}\neq 0$, then $\Lambda_{kt}=0$. However, taking $t=j$, we know $\Lambda_{ij}\neq 0$ but $\Lambda_{kj}\neq 0$, which yields a contradiction. Therefore, $\Lambda$ is a regular matrix.
\end{proof}
    
    \section{Regularity, Unitary Regularity and Groupoid Regularity}
    \label{sec:basic_definitions}

In this section, we introduce the three notions of regularity considered in this paper. The first is \emph{unitary regularity}, in the sense of Dixmier \cite{Dix1954}; the second is \emph{groupoid regularity} \cite{Renaultbook}; and the third is \emph{regularity}, in the sense of Renault (see \cite{Renaultbook}, \cite{Renault2008CartanSI}). We begin by recalling the notion of the normaliser of a $C^*$-subalgebra, which is needed to define these concepts.

Let $\sB \subset \sA$ be an inclusion of $C^*$-algebras. We define the normaliser for arbitrary inclusions. For the groupoid and unitary normalisers, we shall assume that $\sA$ is finite-dimensional.

\begin{definition}[Normaliser]
The \emph{normaliser} of $\sB$ in $\sA$ is
\[
\nor{\sA}(\sB)
=
\{\, n \in \sA \mid n\sB n^*  \cup n^*\sB n \subseteq \sB \,\}.
\]
It is easily seen that this set is a norm closed $*$-semigroup containing $\sB$.
\end{definition}

Assume henceforth that $\sA$ is finite-dimensional.

\begin{definition}[Groupoid Normaliser]
The \emph{groupoid normaliser} of $\sB$ in $\sA$ is
\[
\grnor{\sA}(\sB)
=
\{\, v \in \nor{\sA}(\sB) \mid vv^*v = v \,\},
\]
that is, the set of partial isometries in $\nor{\sA}(\sB)$. 
\end{definition}

\begin{remark}
Note that if $\sB$ is abelian then $\grnor{\sA} (\sB)$ is an \emph{inverse semigroup} with adjoint serving the inverse operation (see \cite[Proposition 2.13]{cartan_triples}).
\end{remark}

\begin{definition}[Unitary Normaliser]
The \emph{unitary normaliser} of $\sB$ in $\sA$ is
\[
\unor{\sA}(\sB)
=
\{\, u \in \sU(\sA) \mid u\sB u^* = \sB \,\}.
\]
\end{definition}

Note that $\unor{\sA} (\sB) = \sU(\sA) \cap \nor{\sA}(\sB) \subseteq \grnor{\sA}(\sB)$.

\begin{prop}[{\cite[Remark 2.2]{cartan_triples}}]\label{prop:abel-equi-reg}
    Let $\mathcal{B}\subset \mathcal{A}$ be a inclusion of finite-dimensional $C^*$-algebras with $\sB$ abelian. Then 
    \[
    \spn (\unor{\sA} (\sB)) = \spn (\grnor{\sA}(\sB)).
    \]
\end{prop}

In what follows, by a \emph{unital inclusion} we mean an inclusion $\sB \subseteq \sA$ such that $1_\sB = 1_\sA$.

\begin{prop}\label{prop:polar-normaliser}
Let $\mathcal{B}\subset \mathcal{A}$ be a unital inclusion of finite-dimensional $C^*$-algebras. If $x\in \nor{\sA}(\mathcal{B})$ has polar decomposition
\[
x=u|x|,
\]
then $u\in \nor{\sA}(\mathcal{B})$ and $|x|\in \mathcal{B}$.
\end{prop}

\begin{proof}
Since $\mathcal{B}$ is a unital $C^*$-subalgebra of $\mathcal{A}$ containing $1_{\mathcal{A}}$, it follows immediately that $|x|\in \mathcal{B}$ whenever $x\in \nor{\sA}(\mathcal{B})$.

Now let $x\in \nor{\sA}(\mathcal{B})$ and write its polar decomposition as $x=u|x|$.

\noindent\textbf{Case 1.} Suppose that $x$ is invertible. Then $|x|$ is invertible and
\[
u=x|x|^{-1}.
\]
Since $x\in \nor{\sA}(\mathcal{B})$ and $|x|^{-1}\in \mathcal{B}$, we conclude that $u\in \nor{\sA}(\mathcal{B}).$

\medskip

\noindent\textbf{Case 2.} Suppose that $x$ is not invertible. For each $n\in\mathbb{N}$, define
\[
u_n=x\left(\frac{1}{n}+|x|\right)^{-1}.
\]
Since $\left(\frac{1}{n}+|x|\right)^{-1}\in \mathcal{B}$, we have $u_n\in \nor{\sA}(\mathcal{B}).$ for every $n$.

As $\mathcal{A}$ is finite dimensional,
\[
|x|\left(\frac{1}{n}+|x|\right)^{-1}
\longrightarrow s(|x|)=u^*u
\]
in norm. Therefore,
\[
u_n
=x\left(\frac{1}{n}+|x|\right)^{-1}
=u|x|\left(\frac{1}{n}+|x|\right)^{-1}
\longrightarrow uu^*u=u
\]
in norm. Since $\nor{\sA}(\mathcal{B})$ is norm closed, it follows that $u\in \nor{\sA}(\mathcal{B}).$
This completes the proof.
\end{proof}

\begin{definition}\label{def:regularity}
Let $\mathcal{B}$ be a $C^*$-subalgebra of $\mathcal{A}$. We say that
$\mathcal{B}$ is \emph{regular} in $\mathcal{A}$ if the $C^*$-algebra
generated by $\nor{\mathcal{A}}(\mathcal{B})$ coincides with
$\mathcal{A}$; that is,
\[
C^*(\nor{\mathcal{A}}(\mathcal{B}))=\mathcal{A}.
\]

We say that $\mathcal{B}$ is \emph{unitary regular} if
\[
C^*(\unor{\mathcal{A}}(\mathcal{B}))=\mathcal{A},
\]
and \emph{groupoid regular} if
\[
C^*(\grnor{\mathcal{A}}(\mathcal{B}))=\mathcal{A}.
\]
\end{definition}

\begin{remark}
In a finite-dimensional setting, the $C^*$-algebra generated by a subset is simply the $C^*$-closure of the algebra generated by that subset. Because $\nor{\sA}(\mathcal{B})$ is closed under multiplication and involution, its linear span is already a $\ast$-algebra. In finite dimensions, this $\ast$-algebra is automatically closed.
\end{remark}

\begin{prop}\label{prop:autmorph_pres_nor}
Let $\sB \subseteq \sA$ be a unital inclusion of (not necessarily finite-dimensional) $C^*$-algebras, and let $\sigma \in \textrm{Aut}(\sA)$. Then
\[
\sigma\bigl(\unor{\sA}(\sB)\bigr)=\unor{\sA}(\sigma(\sB)), \qquad
\sigma\bigl(\grnor{\sA}(\sB)\bigr)=\grnor{\sA}(\sigma(\sB)), \qquad
\sigma\bigl(\nor{\sA}(\sB)\bigr)=\nor{\sA}(\sigma(\sB)).
\]
Consequently, $\sB$ is regular (respectively, unitary regular or groupoid regular) if and only if $\sigma(\sB)$ is regular (respectively, unitary regular or groupoid regular).
\end{prop}

\begin{theorem}\label{thm:reg-iff-gr-reg}
Let $\mathcal{B} \subseteq \mathcal{A}$ be a unital inclusion of finite-dimensional
$C^*$-algebras. Then the following are equivalent:
\begin{enumerate}
    \item[(i)] $\mathcal{B}$ is regular in $\mathcal{A}$.
    \item[(ii)] $\mathcal{A} = \spn\!\bigl(\grnor{\mathcal{A}}(\mathcal{B})\bigr)$.
\end{enumerate}
In particular, $\mathcal{B}$ is regular in $\mathcal{A}$ if and only if it is groupoid regular in $\mathcal{A}$.
\end{theorem}
\begin{proof}
First, observe that $\nor{\sA}(\mathcal{B})$ is closed under scalar multiplication, involution and multiplication. Thus, $\spn(\nor{\sA}(\mathcal{B}))$ is a $*$-subalgebra of $\mathcal{A}$. Because $\mathcal{A}$ is finite-dimensional, every $*$-subalgebra is automatically norm-closed. Therefore, $C^*(\nor{\sA}(\mathcal{B})) = \spn(\nor{\sA}(\mathcal{B}))$.

$(\impliedby)$ Assume $\mathcal{A} = \spn(\grnor{\sA} (\sB))$. By definition, $\grnor{\sA} (\sB) \subseteq \nor{\sA}(\mathcal{B})$. Therefore, we have the inclusions:
\[
\mathcal{A} = \spn(\grnor{\sA} (\sB)) \subseteq \spn(\nor{\sA}(\mathcal{B})) = C^*(\nor{\sA}(\mathcal{B})) \subseteq \mathcal{A}.
\]
Thus, $C^*(\nor{\sA}(\mathcal{B})) = \mathcal{A}$, which means $\mathcal{B}$ is regular in $\mathcal{A}$.

$(\implies)$ Assume $\mathcal{B}$ is regular in $\mathcal{A}$, meaning $C^*(\nor{\sA}(\mathcal{B})) = \spn ( \nor{\sA}(\sB)) = \mathcal{A}$.

Let $x \in \nor{\sA}(\mathcal{B})$. By proposition \ref{prop:polar-normaliser}, its polar decomposition $x = u|x|$ satisfies $u \in \grnor{\sA} (\sB)$ and $|x| \in \mathcal{B}$.

Since $\mathcal{B}$ is spanned by its unitaries, there exist scalars $\lambda_j \in \mathbb{C}$ and unitaries $v_j \in \mathcal{U}(\mathcal{B})$ such that $|x| = \sum_{j=1}^k \lambda_j v_j$. Therefore we get
\[
x = u \left( \sum_{j=1}^k \lambda_j v_j \right) = \sum_{j=1}^k \lambda_j (uv_j).
\]
We now verify that each $uv_j \in \grnor{\sA} (\sB)$. Since $v_j \in \mathcal{U}(\mathcal{B}) \subset \mathcal{B}$, $v_j$ trivially normalises $\mathcal{B}$. Because $\nor{\sA}(\mathcal{B})$ is closed under multiplication, $u \in \nor{\sA}(\mathcal{B})$ and $v_j \in \nor{\sA}(\mathcal{B})$ implies $uv_j \in \nor{\sA}(\mathcal{B})$.

Furthermore, $uv_j$ is a partial isometry:
\[
(uv_j)(uv_j)^*(uv_j) = u v_j v_j^* u^* u v_j = u 1 u^* u v_j = (u u^* u) v_j = u v_j.
\]

Thus, $uv_j \in \grnor{\sA} (\sB)$, which shows that $x \in \spn(\grnor{\sA} (\sB))$. Since this holds for all $x \in \nor{\sA}(\mathcal{B})$, we conclude $\spn(\nor{\sA}(\mathcal{B})) \subseteq \spn(\grnor{\sA} (\sB))$. As $\grnor{\sA} (\sB) \subseteq \nor{\sA}(\mathcal{B})$ is given, equality holds, and therefore $\mathcal{A} = \spn(\grnor{\sA} (\sB))$.

The final assertion follows immediately from the first part together with the elementary observation that
$\mathcal{A}=\spn\!\bigl(\grnor{\mathcal{A}}(\mathcal{B})\bigr)$ implies $\mathcal{A}=C^*\!\bigl(\grnor{\mathcal{A}}(\mathcal{B})\bigr)$.
\end{proof}

Using Proposition \ref{prop:abel-equi-reg} and Theorem \ref{prop:polar-normaliser} we get:
\begin{cor}\label{abelian case}
Let $\mathcal{B} \subseteq \mathcal{A}$ be a unital inclusion of finite-dimensional
$C^*$-algebras, where $\mathcal{B}$ is abelian. Then the following statements are equivalent:
\begin{enumerate}
    \item[(i)] $\mathcal{B}$ is regular in $\mathcal{A}$.
    \item[(ii)] $\mathcal{B}$ is groupoid regular in $\mathcal{A}$.
    \item[(iii)] $\mathcal{B}$ is unitary regular in $\mathcal{A}$.
\end{enumerate}
\end{cor}

We conclude this section with a few examples. The assertions in each case can be verified by direct computation.

\begin{example}\leavevmode\par\label{ex:grop-reg-counter}
\begin{enumerate}
    \item[(i)] The subalgebra $(\C \otimes \mathbb{I}_2) \oplus (\C \otimes \mathbb{I}_2)$ is unitary regular in $\mn{4}$.

    \item[(ii)] The subalgebra $\C \oplus \mn{2}$ is regular in $\mn{3}$, but it is not unitary regular. More generally, for every $n \geq 3$,
    \[
    \C \oplus \mn{n-1} \subseteq \mn{n}
    \]
    is regular but not unitary regular.

    \item[(iii)] The subalgebra $\C \oplus (\C \otimes \mathbb{I}_2)$ is not regular in $\mn{3}$.
\end{enumerate}
\end{example}

    \section{Characterization of Regular Subalgebras of $\mn{n}$}
    \label{sec:char_mn}
    The main purpose of this section is to characterize both regularity and unitary regularity for unital $C^*$-subalgebras of the matrix algebra $\mn{n}$ (see Theorems \ref{thm:regularity-char-mnc} and \ref{thm:uni-reg-char-mnc}, respectively). By Proposition \ref{concrete-inclusion}, every unital $C^*$-subalgebra of $\mn{n}$ is of the form
\[
\bigoplus_{j=1}^k \left( \mn{d_j} \otimes \mathbb{I}_{p_j} \right)
\]
for some vectors $\vec{d}=(d_1,\dots,d_k)$ and $\vec{p}=(p_1,\dots,p_k)$ satisfying $\vec{p}\cdot\vec{d}=n$. Therefore, it suffices to determine when such subalgebras are regular. For groupoid regularity, recall from Theorem \ref{thm:reg-iff-gr-reg} that it is equivalent to regularity in our context.

The preliminary propositions established in this section will play a key role in the proofs of Theorems \ref{thm:regularity-char-mnc} and \ref{thm:uni-reg-char-mnc}. They may also be of independent interest. While Proposition \ref{prop:equal-mult-dim-imply-uni-reg} and Proposition \ref{prop:reg-impy-same-dim} can be deduced from \cite{Bakshi2026}, we offer an independent proof here to maintain the flow of our exposition.

\begin{prop}\label{prop:equal-mult-dim-imply-uni-reg}
    Let $n\in \N$ be given and $d, p, k \in \N$ such that $dpk = n$. Then consider the following unital $C^*$-subalgebra of $\mn{n}$:
    \[
    \Delta_{d, p, k}:=\bigoplus_{1}^{k} \left( \mn{d} \otimes \mathbb{I}_p \right)
    \]
    Then $\Delta_{d, p, k}$ is a unitary regular subalgebra of $\mn{n}$.
\end{prop}
\begin{proof}
    Let $\sA = \mn{n}$, where $n= kdp$. We view $\mn{n}$ as a $k \times k$ block matrices where each block is an element of $\mn{pd}$.

    Under this decompostion, the subalgebra $\sB$ consists of block-diagonal matrices of the form:
    \[
    \sB= \{ \text{diag} (X_1 \otimes \mathbb{I}_p \oplus \dots \oplus X_k \otimes \mathbb{I}_p : X_j \in \mn{d} \}.
    \]
    We will show that the linear span of the unitary normalizer $\unor{\sA} (\sB)$ generates all of $\mn{n}$.

    \noindent{\bf Step 1:} Spanning the Block-Diagonal Matrices. 
    
    First let us find unitary normalizers that are entirely block-diagonal. Consider a unitary matrix $u \in \sA$ of the form:
    $$u = \operatorname{diag}(u_1, u_2, \dots, u_k)$$
    where each $u_j \in \mn{dp}$ is unitary.
    For $u$ to normalize $\mathcal{B}$, we need $u_j (X \otimes I_p) u_j^* \in \mn{d} \otimes I_p$ for every $X \in \mn{d}$.

    To achieve this, we construct $u_j$ as a tensor product of unitaries. Let $V_j \in \sU (\mn{d})$ and $W_j \in \sU(\mn{p})$ be arbitrary unitary matrices. Define:$$u_j = V_j \otimes W_j.$$
    Let us verify this normalizes the $j$-th block:
    $$(V_j \otimes W_j) (X \otimes I_p) (V_j^* \otimes W_j^*) = (V_j X V_j^*) \otimes (W_j I_p W_j^*) = (V_j X V_j^*) \otimes I_p.$$
    Since $V_j X V_j^* \in \mn{d}$, this element remains in $\mn{d} \otimes I_p$. Thus, any block-diagonal unitary $u$ whose blocks are of the form $V_j \otimes W_j$ belongs to $\unor{\sA}(\sB)$.
    Now we take the linear span:
    $$\operatorname{span} \{ V_j \otimes W_j \mid V_j \in \sU(\mn{d}), W_j \in \sU(\mn{p}) \} = \mn{d} \otimes \mn{p} = \mn{dp}.$$
    By taking linear combinations of these block-diagonal normalizers, we can independently generate any arbitrary matrix in each of the $k$ diagonal blocks. Therefore:
    $$\bigoplus_{j=1}^k \mn{dp} \subseteq \operatorname{span}\left(\unor{\sA}(\sB)\right).$$

   \noindent{\bf Step 2: }Spanning the Off-Diagonal Blocks.
    Next, we need to span the cross-terms between different blocks. For any pair of distinct indices $1 \le r < s \le k$, let $S_{rs}$ be the block-swap unitary.
    
    $S_{rs}$ is the matrix that acts as the identity on all blocks except $r$ and $s$, and acts as the standard swap operator $\begin{pmatrix} 0 & I_{dp} \\ I_{dp} & 0 \end{pmatrix}$ on the $r$-th and $s$-th blocks. Let us verify $S_{rs} \in \unor{\sA}(\sB).$ For any $b = \operatorname{diag}(X_1 \otimes I_p, \dots, X_k \otimes I_p) \in \mathcal{B}$:
    $$S_{rs} b S_{rs}^* = \operatorname{diag}(X_1 \otimes I_p, \dots, X_s \otimes I_p, \dots, X_r \otimes I_p, \dots, X_k \otimes I_p).$$
    This simply permutes the $r$-th and $s$-th entries. Because every block of $\mathcal{B}$ has the exact same structure ($\mn{d} \otimes \mathbb{I}_p$), the resulting matrix is still in $\mathcal{B}$. Thus, $S_{rs} \in \unor{\sA}(\sB)$.

    Because the normalizer is closed under multiplication we can multiply the block-diagonal unitaries from earlier by the swap unitaries $S_{rs}$.

    Consider a block-diagonal normalizer $u$ which is $V \otimes W$ in the $r$-th block, and the identity $I_{dp}$ everywhere else. The product $u S_{rs}$ shifts the $V \otimes W$ block into the $(r, s)$ off-diagonal position (and places an identity in the $(s, r)$ position).

    Taking the linear span of these shifted elements over all choices of $V$ and $W$, the $\mn{d} \otimes \mn{p} = \mn{dp}$ span from earlier step is transported perfectly into the $(r, s)$ off-diagonal block.

    Therefore by Step 1, $\operatorname{span}\left(\unor{\sA}(\sB)\right)$ contains all block-diagonal matrices.
    and by Step 2, $\operatorname{span}\left(\unor{\sA}(\sB)\right)$ contains all purely off-diagonal blocks. Since any matrix $M \in \mn{n}$ can be uniquely written as the sum of its block-diagonal part and its block-off-diagonal parts, we conclude that:
    $$\operatorname{span}\left(\unor{\sA}(\sB)\right) = \mn{n}.$$
    Therefore, $\sB$ is a unitarily regular subalgebra of $\mn{n}$. 
\end{proof}

\begin{prop}\label{prop:reg-impy-same-dim}
Let $n \in \N$ and let $d_1, \dots, d_k, p \in \N$ satisfy
\[
n = (d_1 + \cdots + d_k) p,
\]
where $d_i \neq d_j$ for some $i \neq j$. Then 
\[
\bigoplus_{j = 1}^k \left( \mn{d_j} \otimes \mathbb{I}_p \right)
\] 
is not a unitary regular subalgebra (but regular; see Proposition \ref{prop:equal-multi-imply-reg}) of $\mn{n}$.
\end{prop}

\begin{proof}
    First see that the center of $\sB$, denoted $\sZ(\sB)$, consists of the scalar multiples of the identity within each block. It is spanned by the minimal central projections $z_1, z_2, \dots, z_k$, where $z_j$ is the identity matrix of the $j$-th block padded with zeros elsewhere.

    When viewed as matrices in $\sA = \mn{n}$, the rank of the $j$-th central projection is the size of its corresponding block:
    \[
    \operatorname{rank} (z_j) = d_j p, \;( 1 \le j \le k).
    \]
    Let $u \in \unor{\sA}(\sB)$. Then $u\sB u^* = \sB$. The map $\operatorname{Ad}_u : x \mapsto u x u^*$ is a $*$-automorphism of $\sB$, and therefore it must map the center $\sZ(\mathcal{B})$ onto itself. Specifically, $\operatorname{Ad}_u$ must permute the minimal central projections.

    Thus, there exists a permutation $\sigma$ of the indices $\{1, \dots, k\}$ such that:$$u z_j u^* = z_{\sigma(j)} \quad \text{for all } j \in \{1, \dots, k\}$$
    Because $u$ is a unitary matrix, conjugation by $u$ preserves the rank of any matrix. Therefore:$$
    d_jp = \operatorname{rank}(z_j) = \operatorname{rank}(u z_j u^*) = \operatorname{rank}(z_{\sigma(j)}) = d_{\sigma(j)}p,\; (1\le j \le k).
    $$
    Therefore we get:
    \begin{equation}\label{eqn:dj-equal}
        d_j = d_{\sigma(j)},\; (1\le j \le k)
    \end{equation}
    This means a unitary normalizer can only map the $j$-th central projection to the $i$-th central projection if their underlying blocks have the exact same internal dimension ($d_j = d_i$).

    Now by hypothesis, there exist some indices $a \neq b$ such that $d_a \neq d_b$.

    For our unitary normalizer $u$, we know $u z_b = z_{\sigma(b)} u$. Multiplying on the left by $z_a$, we get:$$z_a u z_b = z_a z_{\sigma(b)} u$$
    We also know from (\ref{eqn:dj-equal}), $d_{\sigma(b)} = d_b$. 

    Since $d_a \neq d_b$, it is impossible for $\sigma(b)$ to equal $a$. Because distinct minimal central projections are orthogonal, $z_a z_{\sigma(b)} = 0$. Consequently:$$z_a u z_b = 0.$$
    Because this holds for every unitary $u \in \unor{\sA}(\sB)$, it also holds for any linear combination of such unitaries. Thus, for any $x \in \spn \left( \unor{\sA}(\sB) \right)$, we must have:$$z_a x z_b = 0$$

    The set of all matrices in $\mn{n}$ where $z_a x z_b = 0$ is a strictly smaller subset of $\mn{n}$. Specifically, it misses the entire $(a, b)$-th off-diagonal block, which is isomorphic to the non-zero space $\mn{d_a p \times d_b p}$. Therefore:
    \[
    \spn \left( \unor{\sA}(\sB) \right)\subsetneq \mn{n} = \sA.
    \]
    Thus $\sB$ is not unitary regular in $\sA$.
\end{proof}

\begin{prop}\label{prop:reg-imply-equal-mult}
Let $n \in \N$ and let $d, p_1, \dots ,p_k \in \N$ satisfy
\[
n = d(p_1 + \cdots + p_k),
\]
where $p_i \neq p_j$ for some $i \neq j$. Then 
\[
\bigoplus_{j = 1}^k \left( \mn{d} \otimes \mathbb{I}_{p_j} \right)
\] 
is not a (unitary) regular subalgebra of $\mn{n}$.
\end{prop}

\begin{proof}
    We will show that the $\spn \left(\grnor{\sA}(\sB)\right) \subsetneq \sA$.

    The center of $\mathcal{B}$, denoted $\sZ(\mathcal{B})$, is spanned by the minimal central projections $z_1, \dots, z_k$, where $z_j$ is the identity element of the $j$-th block.

    For any projection $P \in \mathcal{B}$, we can decompose it as $P = \sum_{j=1}^k P z_j$. Because $P z_j$ belongs to $\mn{d} \otimes I_{p_j}$, it must be of the form:$$P z_j = Q_j \otimes I_{p_j}$$for some projection $Q_j \in \mn{d}$.

    If we let $q_j = \operatorname{rank}(Q_j)$, then $Pz_j$ has two important properties:
    \begin{itemize}
        \item[(i)] The corner algebra $(P z_j) \mathcal{B} (P z_j)$ is spatially isomorphic to $\mn{q_j} \otimes I_{p_j}$, which as an abstract $C^*$-algebra is isomorphic to $\mn{q_j}$ (see Proposition \ref{prop:abs-iso-concrete}).
        \item[(ii)] When viewed as a matrix in $\mn{n}$, its rank is $\operatorname{rank}(P z_j) = q_j p_j$
    \end{itemize}

    Now by hypothesis, there exist indices $a \neq b$ such that $p_a \neq p_b$.

    Let $v \in \grnor{\sA} (\sB)$ be an arbitrary partial isometry in the groupoid normalizer. We want to evaluate the component of $v$ that maps the $a$-th block to the $b$-th block, which is given by $w = z_b v z_a$.

    We claim that $w \in \grnor{\sA} (\sB)$. For any $x \in \mathcal{B}$:$$w x w^* = (z_b v z_a) x (z_a v^* z_b) = z_b \left( v (z_a x z_a) v^* \right) z_b$$

    Because $z_a \in \sZ(\mathcal{B})$, we have $z_a x z_a \in \mathcal{B}$. Because $v \in \grnor{\sA} (\sB)$, $v(z_a x z_a)v^* \in \mathcal{B}$. Finally, multiplying by $z_b \in \sZ(\mathcal{B})$ keeps the element in $\mathcal{B}$. Thus, $w \mathcal{B} w^* \subseteq \mathcal{B}$. By a symmetric argument, $w^* \mathcal{B} w \subseteq \mathcal{B}$.

    Therefore $w \in \grnor{\sA} (\sB)$. Consequently, the initial and final projections of $w$ must belong to $\mathcal{B}$.

    Its initial projection satisfies $w^*w = z_a v^* z_b v z_a \le z_a$. Thus, $w^*w = Q_a \otimes I_{p_a}$ for some projection $Q_a \in \mn{d}$. Let $q = \operatorname{rank}(Q_a)$.

    Its final projection satisfies $ww^* = z_b v z_a v^* z_b \le z_b$. Thus, $ww^* = R_b \otimes I_{p_b}$ for some projection $R_b \in \mn{d}$. Let $r = \operatorname{rank}(R_b)$.

    Now the map $\alpha(x) = w x w^*$ is a $C^*$-algebra isomorphism from the initial corner $(w^*w)\mathcal{B}(w^*w)$ onto the final corner $(ww^*)\mathcal{B}(ww^*)$. The initial corner is isomorphic to $\mn{q}$ and the final corner is isomorphic to $\mn{r}$.
    
    Because these two matrix algebras are isomorphic, they must have the same internal dimension:$$q = r.$$
    At the same time, because $w$ is a partial isometry in $\mn{n}$, its initial and final projections must have the exact same rank in $\mn{n}$:$$\operatorname{rank}(w^*w) = \operatorname{rank}(ww^*) \implies q p_a = r p_b$$
    substituting $q = r$ into this rank equation yields:
    $$ q(p_a - p_b) = 0.$$

    Because $p_a \neq p_b$, we have $p_a - p_b \neq 0$. This forces $q = 0$.
    If $q = 0$, then the initial projection $w^*w = 0$, which implies $w = 0$.

    Recall that $w = z_b v z_a$. We have shown that for any $v \in \grnor{\sA}(\sB)$, the off-diagonal block $z_b v z_a$ is strictly zero.

    By linearity, for any element $X \in \operatorname{span}\left(\grnor{\sA}(\sB)\right)$, we must have:$$z_b X z_a = 0$$
    However, in $\mathcal{A} = \mn{n}$, the subspace $z_b \mathcal{A} z_a$ is isomorphic to $\mn{d p_b \times d p_a}$ and is non-zero.
    
    Therefore:
    \[
    \spn \left(\grnor{\sA}(\sB)\right) \subsetneq \sA.
    \]
\end{proof}

In a similar argument as given in Proposition \ref{prop:reg-imply-equal-mult} we can have the following generalization:

\begin{prop}\label{prop:reg-imply-equal-mult-gen}
Let $n \in \N$ and let $d_1,\dots,d_k, p_1, \dots ,p_k \in \N$ satisfy
\[
n = \sum_{j=1}^k d_j p_j,
\]
where $p_i \neq p_j$ for some $i \neq j$. Then 
\[
\bigoplus_{j = 1}^k \left( \mn{d_j} \otimes \mathbb{I}_{p_j} \right)
\] 
is not a regular subalgebra of $\mn{n}$.
\end{prop}
\begin{proof}
    We will argue in the similar line as that of Proposition \ref{prop:reg-imply-equal-mult}.

    By hypothesis, there exist indices $a \neq b$ such that $p_a \neq p_b$. We want to show that:
    \[
    \spn \left(\grnor{\sA}(\sB)\right) \subsetneq \sA.
    \]

    The center of $\mathcal{B}$ is spanned by the minimal central projections $z_1, \dots, z_k$, where $z_j = I_{d_j} \otimes I_{p_j}$ (padded with zeros outside the $j$-th block).

    Let $v \in \grnor{\sA}(\sB)$. 

    We isolate the component of $v$ that intertwines the $a$-th block and the $b$-th block by defining $w = z_b v z_a$.
    By the exact same algebraic verification as in Proposition \ref{prop:reg-imply-equal-mult} (since $z_a, z_b \in \sZ(\mathcal{B})$), we have $w \in \grnor{\sA} (\sB)$.

    Because $w$ normalizes $\mathcal{B}$, its initial projection $w^*w$ and final projection $ww^*$ must belong to $\mathcal{B}$ and we have:
    
    $w^*w = z_a v^* z_b v z_a \le z_a$. Thus, it is contained entirely within the $a$-th block, meaning $w^*w = Q_a \otimes I_{p_a}$ for some projection $Q_a \in \mn{d_a}$. Let $q = \operatorname{rank}(Q_a)$. (Note that $q \le d_a$).

    $ww^* = z_b v z_a v^* z_b \le z_b$. Thus, it is contained entirely within the $b$-th block, meaning $ww^* = R_b \otimes I_{p_b}$ for some projection $R_b \in \mn{d_b}$. Let $r = \operatorname{rank}(R_b)$. (Note that $r \le d_b$).

    The map $\alpha(x) = w x w^*$ is a $C^*$-algebra isomorphism from the initial corner to the final corner:$$(Q_a \otimes I_{p_a}) \mathcal{B} (Q_a \otimes I_{p_a}) \cong (R_b \otimes I_{p_b}) \mathcal{B} (R_b \otimes I_{p_b}).$$ 
    As abstract $C^*$-algebras, these corners are isomorphic to $Q_a \mn{d_a} Q_a$ and $R_b \mn{d_b} R_b$, which in turn are isomorphic to $\mn{q}$ and $\mn{r}$, respectively. For $\mn{q}$ to be abstractly isomorphic to $\mn{r}$, we must have:$$q = r.$$

    On the other hand because $w$ is a partial isometry in $\mn{n}$, its initial and final projections must have the same rank as matrices in $\mn{n}$:$$\operatorname{rank}(w^*w) = \operatorname{rank}(ww^*) \implies q p_a = r p_b.$$

    Now substituting $q = r$ into this last equation we get:
    $$q p_a = q p_b \implies q (p_a - p_b) = 0.$$
    By hypothesis, $p_a \neq p_b$, so $p_a - p_b \neq 0$. This immediately forces $q = 0$. If $q = 0$, then $w^*w = 0$, which means $w = 0$.

    Now we can conclude by the same logic as in the end of the proof of Proposition \ref{prop:reg-imply-equal-mult}.
\end{proof}

\begin{prop}\label{prop:equal-multi-imply-reg}
Let $n \in \N$ and let $d_1, \dots, d_k, p \in \N$ satisfy $n = (d_1 + \cdots + d_k) p$. Then 
\[
\bigoplus_{j = 1}^k \left( \mn{d_j} \otimes \mathbb{I}_p \right)
\] 
is a regular subalgebra of $\mn{n}$.
\end{prop}
\begin{proof}
    We will prove that $\sB$ is groupoid regular in $\sA$ (see Theorem \ref{thm:reg-iff-gr-reg}).
    
    Let $N = \sum_{j=1}^k d_j$. We can naturally identify the algebra $\sA = \mn{n}$ with the tensor product:$$\mathcal{A} = \mn{N} \otimes \mn{p}.$$
    Under this identification, let $\sC \subseteq \mn{N}$ be the standard block-diagonal subalgebra:$$\sC = \bigoplus_{j=1}^k \mn{d_j}.$$
    
    Our algebra $\mathcal{B}$ is then exactly the tensor product of $\mathcal{C}$ with the $p \times p$ identity matrix:$$\mathcal{B} = \mathcal{C} \otimes I_p = \{ C \otimes I_p \mid C \in \mathcal{C} \}$$

    Let $e_{x,y}$ for $1 \le x, y \le N$ denote the standard matrix units of $\mn{N}$. Crucially, because $\mathcal{C}$ contains the full matrix algebra for each of its blocks, it contains the entire diagonal of $\mn{N}$. Therefore, the diagonal matrix units belong to $\mathcal{C}$:$$e_{x,x} \in \mathcal{C} \quad \text{for all } 1 \le x \le N.$$

    Choose any two indices $x, y \in \{1, \dots, N\}$. Let $W \in \mn{p}$ be an arbitrary $p \times p$ unitary matrix ($W W^* = W^* W = I_p$). Define the element $v \in \mathcal{A}$ by:$$v = e_{x,y} \otimes W.$$

    Now let $b = C \otimes I_p$ be an arbitrary element of $\mathcal{B}$. Conjugating $b$ by $v$ yields:$$v b v^* = (e_{x,y} \otimes W) (C \otimes I_p) (e_{y,x} \otimes W^*) = (e_{x,y} C e_{y,x}) \otimes (W I_p W^*)$$

    For any $N \times N$ matrix $C$, the operation $e_{x,y} C e_{y,x}$ isolates the $(y,y)$-th entry of $C$ and places it in the $(x,x)$-th position. Thus:$$e_{x,y} C e_{y,x} = C_{y,y} e_{x,x}$$where $C_{y,y} \in \mathbb{C}$ is a scalar. Substituting this back gives:$$v b v^* = (C_{y,y} e_{x,x}) \otimes I_p = C_{y,y} (e_{x,x} \otimes I_p).$$ Because $e_{x,x} \in \mathcal{C}$, this element belongs to $\mathcal{C} \otimes I_p = \mathcal{B}$. Thus, $v \mathcal{B} v^* \subseteq \mathcal{B}$.

    By an identical calculation using $v^* b v$, we find:$$v^* b v = C_{x,x} (e_{y,y} \otimes I_p) \in \mathcal{B}$$Thus, $v^* \mathcal{B} v \subseteq \mathcal{B}$.

    This proves that $v = e_{x,y} \otimes W$ is a partial isometry in the groupoid normalizer $\grnor{\sA} (\sB)$.

    Thus for any $1 \le x, y \le N$ and any unitary $W \in \mn{p}$, the element $e_{x,y} \otimes W$ belongs to $\grnor{\sA}(\sB)$.

    Fix the indices $x$ and $y$. As $W$ ranges over all unitary matrices in $\mn{p}$, its linear span generates the entirety of $\mn{p}$. Therefore:
    $$\spn \{ e_{x,y} \otimes W \mid W \text{ is unitary} \} = e_{x,y} \otimes \mn{p}.$$ Because this holds for every pair of indices $(x,y)$, taking the span over all $1 \le x, y \le N$ yields the sum of all such blocks:

    $$\spn\left(\grnor{\sA}(\sB)\right) \supseteq \spn \{ e_{x,y} \otimes \mn{p} \mid 1 \le x, y \le N \} = \mn{N} \otimes \mn{p}.$$

    Since $\mn{N} \otimes \mn{p} = \mathcal{A}$, we conclude that:
    \[
    \spn\left(\grnor{\sA}(\sB)\right)=\sA.
    \]
    This completes the proof.
\end{proof}

\begin{theorem}[Regular Subalgebras of $\mn{n}$]\label{thm:regularity-char-mnc}
Let $\sB \subseteq \mn{n}$ be a unital $C^*$-subalgebra. Up to unitary conjugation in $\mn{n}$, there exist $d_1, \dots, d_k, p_1, \dots, p_k \in \N$ with
\[
\sum_{j=1}^{k} p_j d_j = n
\]
such that
\[
\sB = \bigoplus_{j=1}^{k} \left( \mn{d_j} \otimes \mathbb{I}_{p_j} \right).
\]
Then $\sB$ is regular if and only if
\[
p_1 = \cdots = p_k.
\]
\end{theorem}

\begin{proof}
    \noindent $(\impliedby)$. This follows from Proposition \ref{prop:equal-multi-imply-reg}.

    \vspace{0.5cm}
    \noindent $(\implies)$. Follows from Proposition \ref{prop:reg-imply-equal-mult-gen}.
\end{proof}

\begin{theorem}[Unitary Regular Subalgebras in $\mn{n}$]\label{thm:uni-reg-char-mnc}
Let $\sB \subseteq \mn{n}$ be a unital subalgebra. Up to unitary conjugation in $\mn{n}$, there exist $d_1, \dots, d_k, p_1, \dots, p_k \in \N$ with
\[
\sum_{j=1}^{k} d_j p_j = n
\]
such that
\[
\sB = \bigoplus_{j=1}^{k} \left( \mn{d_j} \otimes \mathbb{I}_{p_j} \right).
\]
Then $\sB$ is unitary regular if and only if
\[
p_1 = \cdots = p_k
\quad \text{and} \quad
d_1 = \cdots = d_k.
\]
\end{theorem}

\begin{proof}
    \noindent $(\impliedby)$. This follows from Corollary \ref{prop:equal-mult-dim-imply-uni-reg}. 

    \vspace{0.5cm}
    \noindent $(\implies)$. From Proposition \ref{prop:reg-impy-same-dim} it follows that $d_1 = \dots = d_k$. On the other hand if $p_1 = \dots = p_k$ doesn't hold then by Theorem \ref{thm:regularity-char-mnc} $\sB$ cannot be even regular in $\sA$ and consequently cannot be unitary regular. 

    This completes the proof.
\end{proof}

\begin{remark}\label{rem:abs_vs_concrete}
Let us consider the abstract $C^*$-algebra $\C^2:=\{ (\lambda, \mu): \lambda, \mu \in \C \}$. Observe that $\C^2$ admits (at least) two distinct unital embeddings into $\mn{6}$, namely,
\[
\C^2 \cong(\C \otimes \mathbb{I}_3)\oplus (\C \otimes \mathbb{I}_3) \subseteq \mn{6},
\qquad \text{and} \qquad
\C^2 \cong(\C \otimes \mathbb{I}_4)\oplus (\C \otimes \mathbb{I}_2) \subseteq \mn{6}.
\]
By Theorems~\ref{thm:regularity-char-mnc} and~\ref{thm:uni-reg-char-mnc}, the first embedding is unitarily regular (and hence regular), whereas the second embedding is not even regular. This demonstrates that (unitary) regularity is fundamentally a representation-dependent notion: it depends on the specific spatial embedding of the smaller algebra into the larger one. Consequently, when discussing the (unitary) regularity of an abstract $C^*$-algebra, one must always specify the embedding under consideration.
\end{remark}

\begin{cor}[Regular MASAs of $\mn{n}$,\, {cf. \cite[Proposition 3.7]{Bakshi2026}}]\label{cor:reg_masa_mn}
    Let $\sD$ be an abelian subalgebra of $\mn{n}$. Then $\sD$ is regular if and only if, up to unitary conjugation, it is of the form
    \[
    \Delta_p
    :=
    \bigoplus_{1}^{n/p}
    \left(
    \C \otimes \mathbb{I}_p
    \right),
    \]
    for some divisor $p \mid n$. In particular, $\Delta_p$ is a regular MASA if and only if $p = 1$.
\end{cor}

\begin{definition}[Homogeneous $C^*$-algebras]
Let $\sB$ be an abstract $C^*$-algebra, not necessarily finite dimensional. We say that $\sB$ is \emph{$d$-homogeneous} if every non-zero inquivalent irreducible representation of $\sB$ has dimension $d$. The algebra $\sB$ is said to be \emph{homogeneous} if it is $d$-homogeneous for some $d \in \N$.
\end{definition}

\begin{prop}
A finite-dimensional $C^*$-algebra $\sA$ is homogeneous if and only if there exist $r,n \in \N$ such that
\[
\sA \cong \bigoplus_{j=1}^{r} \mn{n}.
\]
\end{prop}

\begin{proof}
\noindent $(\implies)$.
By Proposition \ref{thm:artin-wedderburn}, suppose
\[
\sA \cong \bigoplus_{j=1}^{k} \mn{d_j}.
\]
Now by Proposition~\ref{prop:rep_of_fin_alg}, the coordinate projections
\[
\pi_j : \sA \longrightarrow \mn{d_j}, \qquad j=1,\ldots,k,
\]
constitute, up to unitary equivalence, the complete list of non-zero irreducible representations of $\sA$. Since $\sA$ is $d$-homogeneous, every irreducible representation has dimension $d$. Hence,
\[
d_j=\dim(\pi_j)=d,\qquad j=1,\ldots,k,
\]
and therefore
\[
\sA \cong \bigoplus_{j=1}^{k} \mn{d}.
\]

\medskip

\noindent $(\impliedby)$.
Conversely, suppose
\[
\sA \cong \bigoplus_{j=1}^{r} \mn{n}.
\]
Again, by Proposition~\ref{prop:rep_of_fin_alg}, the coordinate projections
\[
\pi_j : \sA \longrightarrow \mn{n}, \qquad j=1,\ldots,r,
\]
are, up to unitary equivalence, the only non-zero irreducible representations of $\sA$. Since
\[
\dim(\pi_j)=n,\qquad j=1,\ldots,r,
\]
it follows that every non-zero irreducible representation of $\sA$ has dimension $n$. Thus $\sA$ is $n$-homogeneous.
\end{proof}

\begin{prop}\label{homogenity}
    Let $\mathcal{B}\subseteq \mn{n}$ be a unital inclusion. We know by Proposition \ref{Spatial-inclusion} there exist $d_1, \dots, d_k, p_1, \dots, p_k \in \N$ with
    \[
    \sum_{j=1}^k d_j p_j = n,
    \]
such that
    \[
    \mathcal{B}
    \cong
    \bigoplus_{j=1}^k \left(\mn{d_j} \otimes \mathbb{I}_{p_j}\right).
    \]
Then, $\mathcal{B}$ is a homogeneous $C^*$-subalgebra of $\mathcal{A}$ if and only if $d_1= \dots= d_k$.
\end{prop}
\begin{proof}
\noindent ($\implies$). By proposition \ref{prop:abs-iso-concrete},
\[
    \mathcal{B}
    \cong
    \bigoplus_{j=1}^k \left(\mn{d_j} \otimes \mathbb{I}_{p_j}\right)
    \cong \bigoplus_{j=1}^{k} \mn{d_j}.
    \]

Now by Proposition~\ref{prop:rep_of_fin_alg}, the coordinate projections
\[
\pi_j : \sB \longrightarrow \mn{d_j}, \qquad j=1,\ldots,k,
\]
constitute, up to unitary equivalence, the complete list of non-zero irreducible representations of $\sB$. Since $\sB$ is homogeneous, say $d$-homogeneous, every irreducible representation has dimension $d$. Hence,
\[
d_j=\dim(\pi_j)=d,\qquad j=1,\ldots,k.
\]

\vspace{0.2cm}
    \noindent ($\impliedby$). If $d_1 = d_2 = \dots = d_k = d$, say, then $$\sB \cong \bigoplus_{j=1}^k (\mn{d}\otimes \mathbb{I}_{p_j})\cong \bigoplus_{1}^k \mn{d}.$$
    Now by previous proposition $\sB$ is homogeneous.
 
\end{proof}

By Corollary \ref{abelian case} we know that for the abelian subalgebra in $\mn{n}$ all three type of regularities are equivalent. Now by Theorem \ref{thm:regularity-char-mnc} and Proposition \ref{homogenity}, we have a more general version of it.

\begin{cor}\label{cor:hom-imply-reg-coincide}
    Let $\mathcal{B}\subseteq \mn{n}$ be a homogeneous subalgebra. Then all three types of regularities are equivalent.  
\end{cor}

Theorems \ref{thm:regularity-char-mnc} and \ref{thm:uni-reg-char-mnc} provide a complete characterization of all (unitary) regular subalgebras, which we summarize in the table below (see table (\ref{tab:tuple_pairs}) below) in $\mn{n}$ upto $n=4$.
\begin{table}[htbp]
\caption{Enumeration of tuple pairs $(\vec{p},\vec{d})$ for $\sum_{j=1}^k p_j d_j = n$}
\label{tab:tuple_pairs}
\renewcommand{\arraystretch}{1.5}
\begin{tabular}{cll}
\hline\hline
$n$
&
\multicolumn{1}{c}{\begin{tabular}{@{}c@{}}
\texttt{Regularity}\\
$p_1=\dots=p_k$
\end{tabular}}
&
\multicolumn{1}{c}{\begin{tabular}{@{}c@{}}
\texttt{Unitary Regularity}\\
$p_1=\dots=p_k$\\
and $d_1=\dots=d_k$
\end{tabular}}
\\
\hline
1 
& \begin{tabular}{@{}l@{}} 
    $p=(1), \quad d=(1)$\\
  \end{tabular} 
& \begin{tabular}{@{}l@{}} 
    $p=(1), \quad d=(1)$ 
  \end{tabular} \\
\hline
2 
& \begin{tabular}{@{}l@{}} 
    $p=(1), \quad d=(2)$ \\ 
    $p=(2), \quad d=(1)$ \\ 
    $p=(1, 1), \quad d=(1, 1)$ 
  \end{tabular} 
& \begin{tabular}{@{}l@{}} 
    $p=(1), \quad d=(2)$ \\ 
    $p=(2), \quad d=(1)$ \\ 
    $p=(1, 1), \quad d=(1, 1)$ 
  \end{tabular} \\
\hline
3 
& \begin{tabular}{@{}l@{}} 
    $p=(1), \quad d=(3)$ \\ 
    $p=(3), \quad d=(1)$ \\ 
    $p=(1, 1), \quad d=(1, 2)$ \\ 
    $p=(1, 1), \quad d=(2, 1)$ \\ 
    $p=(1, 1, 1), \quad d=(1, 1, 1)$ 
  \end{tabular} 
& \begin{tabular}{@{}l@{}} 
    $p=(1), \quad d=(3)$ \\ 
    $p=(3), \quad d=(1)$ \\ 
    $p=(1, 1, 1), \quad d=(1, 1, 1)$ 
  \end{tabular} \\
\hline
4 
& \begin{tabular}{@{}l@{}} 
    $p=(1), \quad d=(4)$ \\ 
    $p=(2), \quad d=(2)$ \\ 
    $p=(4), \quad d=(1)$ \\ 
    $p=(1, 1), \quad d=(1, 3)$ \\ 
    $p=(1, 1), \quad d=(3, 1)$ \\ 
    $p=(1, 1), \quad d=(2, 2)$ \\ 
    $p=(2, 2), \quad d=(1, 1)$ \\ 
    $p=(1, 1, 1), \quad d=(1, 1, 2)$ \\ 
    $p=(1, 1, 1), \quad d=(1, 2, 1)$ \\ 
    $p=(1, 1, 1), \quad d=(2, 1, 1)$ \\ 
    $p=(1, 1, 1, 1), \quad d=(1, 1, 1, 1)$ 
  \end{tabular} 
& \begin{tabular}{@{}l@{}} 
    $p=(1), \quad d=(4)$ \\ 
    $p=(2), \quad d=(2)$ \\ 
    $p=(4), \quad d=(1)$ \\ 
    $p=(1, 1), \quad d=(2, 2)$ \\ 
    $p=(2, 2), \quad d=(1, 1)$ \\ 
    $p=(1, 1, 1, 1), \quad d=(1, 1, 1, 1)$ 
  \end{tabular} \\
\hline\hline
\end{tabular}
\end{table}

    \section{Regular Subalgebras in General Finite Dimensional Algebras}
    \label{sec:char_gen}
    
In this section, we characterize (unitary) regularity for arbitrary finite-dimensional unital inclusions of $C^*$-algebras. We begin by introducing the notation and conventions that will be used throughout the remainder of this section.

For a tuple $\Vec{n}=(n_1,\dots,n_l)\in\N^l$, we use the notation
\[
\mn{\Vec{n}}
:=
\bigoplus_{i=1}^l \mn{n_i}.
\]

Let $\Vec{n}\in\N^l$ and $\Vec{d}\in\N^k$ be such that there exists a matrix
\[
\Lambda=[\Lambda_{ij}]\in \mathbb{M}_{l\times k}(\N\cup\{0\})
\]
satisfying
\[
\Lambda\Vec{d}=\Vec{n}.
\]
The pair $(\Vec{d},\Lambda)$ determines a unital embedding
\begin{align*}
\Phi_\Lambda:\mn{\Vec{d}}&\longrightarrow \mn{\Vec{n}},\\
x&\longmapsto
\left(\Phi_\Lambda^{(1)}(x),\dots,\Phi_\Lambda^{(l)}(x)\right),
\end{align*}
where, for $x=(x_1,\dots,x_k)\in\mn{\Vec{d}}$,
\[
\Phi_\Lambda^{(i)}(x)
:=
\bigoplus_{j=1}^k
\left(x_j\otimes I_{\Lambda_{ij}}\right),
\qquad i=1,\dots,l.
\]
Here, by convention, a summand with $\Lambda_{ij}=0$ is omitted.

We write
\[
\sB(\Vec{d},\Lambda)
:=
\Phi_\Lambda\!\left(\mn{\Vec{d}}\right)
\subseteq
\mn{\Vec{n}}.
\]

It is immediate that $\sB(\Vec{d},\Lambda)$ is a unital $C^*$-subalgebra of $\mn{\Vec{n}}$. It is known that unital $C^*$-subalgebra of $\mn{\Vec{n}}$ arises in this way:

\begin{prop}[\cite{davidson_book} ]\label{prop:subalg_of_fin_dim_alg}
Let $\Vec{n}=(n_1,\dots,n_l)\in\N^l$. A subset
$
\sB\subseteq\mn{\Vec{n}}
$
is a unital $C^*$-subalgebra if and only if there exist $k\in\N$, a tuple $\Vec{d}\in\N^k$, and a matrix
\[
\Lambda\in\mathbb{M}_{l\times k}(\N\cup\{0\})
\]
satisfying
\[
\Lambda\Vec{d}=\Vec{n},
\]
such that
\[
\sB=\sB(\Vec{d},\Lambda).
\]
\end{prop}

\begin{definition}
    The matrix $\Lambda$ is called \emph{inclusion matrix} of the subalgebra $\sB(\Vec{d}, \Lambda)$.
\end{definition}

\noindent\textbf{Convention.} From now on, let $\vec{d}\in \N^k$, $\vec{n}\in \N^l$, and let $\Lambda$ be an $l \times k$ matrix satisfying $\Lambda \vec{d}=\vec{n}$. We will use the phrase
\[
``\text{the subalgebra }\mn{\vec{d}} \subseteq \mn{\vec{n}} \text{ with inclusion matrix } \Lambda"
\]
to mean the unital $C^*$-subalgebra $\sB(\vec{d}, \Lambda)$ of $\mn{\vec{n}}$.

The main goal of this section is to determine which inclusions in $\mn{\vec{n}}$ are (unitarily) regular. More precisely, we characterize those matrices $\Lambda$ for which the subalgebra $\mn{\vec{d}}$ is (unitarily) regular in $\mn{\vec{n}}$. We begin with an example of an inclusion that is not regular.

\begin{example}
    Consider the following unital inclusion: $$\Lambda = \begin{pmatrix} 1 & 0 \\ 0 & 1 \\ 1 & 1 \end{pmatrix}, \quad \vec{d} = \begin{pmatrix} 2 \\ 4 \end{pmatrix} \text{ and } \vec{n} = \begin{pmatrix} 2 \\ 4 \\ 6 \end{pmatrix}.$$
   Then $\mn{\vec{d}}\subseteq \mn{\vec{n}}$ is not regular with respect to the inclusion matrix $\Lambda$.
\end{example}
\begin{proof}
   Let the ambient algebra be $\mathcal{A} = \mn{2} \oplus \mn{4} \oplus \mn{6}$. The subalgebra $\sB:=\sB(\vec{d}, \Lambda)$ consists of elements of the form:$$b(X, Y) = \left( X, \quad Y, \quad \begin{pmatrix} X & 0 \\ 0 & Y \end{pmatrix} \right)$$where $X \in \mn{2}$ and $Y \in \mn{4}$. The center of $\mathcal{B}$ is spanned by two minimal central projections, corresponding to the identity elements of the $X$-block and the $Y$-block:$$z_1 = b(I_2, 0) = \left( I_2, \quad 0, \quad \begin{pmatrix} I_2 & 0 \\ 0 & 0 \end{pmatrix} \right)$$$$z_2 = b(0, I_4) = \left( 0, \quad I_4, \quad \begin{pmatrix} 0 & 0 \\ 0 & I_4 \end{pmatrix} \right)$$Notice that $z_1 + z_2 = 1_{\mathcal{A}}$.

   We will show that the groupoid normalizer $\grnor{\sA}(\sB)$ completely misses a subspace of $\mn{6}$.

   Let $v = (v_1, v_2, v_3) \in \mathcal{A}$ be an arbitrary partial isometry in $\grnor{\sA}(\sB)$. Because $z_1, z_2 \in \sZ(\mathcal{B})$, we know that multiplying a normalizer by central projections of $\mathcal{B}$ yields another normalizer. Define:$$w := z_2 v z_1$$By construction, $w \in \grnor{\sA}(\sB)$. Let us compute the components of $w = (w_1, w_2, w_3)$ directly in $\mathcal{A}$:
   $$w_1 = 0 \cdot v_1 \cdot I_2 = 0, \,w_2 = I_4 \cdot v_2 \cdot 0 = 0,\, w_3 = \begin{pmatrix} 0 & 0 \\ 0 & I_4 \end{pmatrix} v_3 \begin{pmatrix} I_2 & 0 \\ 0 & 0 \end{pmatrix}$$
   Thus, $w = (0, 0, w_3)$. Inside the $\mn{6}$ summand, $w_3$ represents exactly the bottom-left $4 \times 2$ block of the matrix $v_3$.

   Because $w \in \grnor{\sA}(\sB)$, its initial projection $w^*w$ must be an element of $\mathcal{B}$. Let us compute $w^*w$ in $\mathcal{A}$:$$w^*w = (0, 0, w_3^*w_3).$$
   Since $w^*w \in \mathcal{B}$, there must exist some $X \in \mn{2}$ and $Y \in \mn{4}$ such that:$$(0, 0, w_3^*w_3) = b(X, Y) = \left( X, \quad Y, \quad \begin{pmatrix} X & 0 \\ 0 & Y \end{pmatrix} \right).$$
   By equating the components across the direct sum, we get: $X = 0$ and $ Y = 0$. If $X=0$ and $Y=0$, then the third component must also be identically zero. Therefore:$$w_3^* w_3 = \begin{pmatrix} 0 & 0 \\ 0 & 0 \end{pmatrix}.$$The only matrix whose initial projection is zero is the zero matrix itself. This implies $w_3 = 0$.
   
   We have shown that for any normalizer $v \in \grnor{\sA}(\sB)$, the off-diagonal projection $z_2 v z_1$ is exactly zero. 
   
   By linearity, for any element $T \in \operatorname{span}\left(\grnor{\sA}(\sB)\right)$, we must have:$$z_2 T z_1 = 0.$$ However, in the full ambient algebra $\mathcal{A}$, the subspace $z_2 \mathcal{A} z_1$ is non-zero. For any $T = (T_1, T_2, T_3) \in \mathcal{A}$:$$z_2 T z_1 = \left( 0, 0, \begin{pmatrix} 0 & 0 \\ 0 & I_4 \end{pmatrix} T_3 \begin{pmatrix} I_2 & 0 \\ 0 & 0 \end{pmatrix} \right)$$ which corresponds to the entire space of $4 \times 2$ matrices in the bottom-left corner of the $\mn{6}$ block. Because every element in the span of $\grnor{\sA}(\sB)$ has a strictly zero bottom-left block in $\mn{6}$, the span does not equal $\mathcal{A}$. Therefore:$$\operatorname{span}\left(\grnor{\sA}(\sB)\right) \subsetneq \mathcal{A},$$ and the inclusion $\mathcal{B} \subseteq \mathcal{A}$ is not regular.
\end{proof}

\begin{lem}\label{lem:reg_pres_direct_sum}
Let $\sB_j \subseteq \sA_j$ be a unital inclusion of finite-dimensional $C^*$-algebras for $j = 1, \dots, m$. Suppose that $\sB_j$ is (unitarily) regular in $\sA_j$ for each $j$. Then
\[
\sB := \bigoplus_{j=1}^m \sB_j
\]
is (unitarily) regular in
\[
\sA := \bigoplus_{j=1}^m \sA_j.
\]
\end{lem}

\begin{remark}
The proof of the Lemma \ref{lem:reg_pres_direct_sum}  is easy but it is to be noted that its converse is false! For example, $\C$ is regular (even unitary regular) in $\mn{2}\oplus \mn{3}$
    with the inclusion: 
    \[
    \Lambda = \begin{pmatrix}
        2 \\ 3
    \end{pmatrix}
    \]
    But it is impossible to write it as a direct sum $\sB_1 \oplus \sB_2$ of unital subalgebras.
\end{remark}

The following elementary lemma will be used repeatedly in the sequel. Its proof is straightforward.
\begin{lem}\label{cut-down-regularity}
    Let $\mathcal{B}\subseteq \mathcal{A}$ be an inclusion of finite dimensional $C^*$-algebras and $p$ be a central projection in $\mathcal{A}$. Then $p\grnor{\sA}(\sB)\subseteq \grnor{p\sA}(p\sB)$ and $p\unor{\sA}(\sB) \subseteq \unor{p\sA}(p\sB)$. In particular, if $\sB\subseteq \sA$ is regular (respectively, unitary regular) then $p\sB \subseteq p\sA$ is also regular (respectively, unitary regular).
\end{lem}

\begin{prop}\label{prop:reg_imply_col_cond}
Suppose $\mn{\vec{d}}$ is regular in $\mn{\vec{n}}$ with inclusion matrix $\Lambda$. If there exists an index $i$ and distinct indices $a,b$ such that $\Lambda_{i,a}>0$ and $\Lambda_{i,b}>0$, then the corresponding columns of $\Lambda$ must be identical; that is,
\[
\vec{\Lambda}_a=\vec{\Lambda}_b.
\]
\end{prop}

\begin{proof}
    Let $z_1, \dots, z_k$ be the minimal central projections of $\mathcal{B}$. By hypothesis, the $a$-th block and $b$-th block both have non-zero multiplicity in the $i$-th summand of $\mathcal{A}$. 

    Because $z_a$ and $z_b$ both have non-zero spatial support inside the $i$-th summand of $\mathcal{A}$, the subspace $z_b \mathcal{A} z_a$ is strictly non-zero.

    Because $z_b \mathcal{A} z_a \neq \{0\}$, the groupoid normalizer must span this space. Therefore, there exists at least one partial isometry $v \in \grnor{\sA}(\sB)$ that has a non-zero component in $z_b \mathcal{A} z_a$. Define:$$w = z_b v z_a$$
    Because $z_a, z_b \in \sZ(\mathcal{B})$ and $v \in \grnor{\sA}(\sB)$, we have $w \in \grnor{\sA}(\sB)$. By our choice of $v$, we know $w \neq 0$.

    Because $w$ normalizes $\mathcal{B}$, its initial projection $P = w^*w$ and final projection $Q = ww^*$ must belong to $\mathcal{B}$.

    Since $w = z_b v z_a$, we have $P = w^*w \le z_a$. Thus, $P$ lives entirely in the $a$-th block of $\mathcal{B}$. As an abstract element of $\mn{d_a}$, $P$ is a projection of some rank $q$.

    Similarly, $Q = ww^* \le z_b$ lives entirely in the $b$-th block of $\mathcal{B}$. As an abstract element of $\mn{d_b}$, $Q$ is a projection of some rank $r$.

    Because $w$ normalizes $\mathcal{B}$, the map $\operatorname{Ad}_w(x) = w x w^*$ defines a $C^*$-algebra isomorphism between the corners:$$P \mathcal{B} P \cong Q \mathcal{B} Q$$Abstractly, $P \mathcal{B} P \cong \mn{q}$ and $Q \mathcal{B} Q \cong \mn{r}$. For these two full matrix algebras to be isomorphic, their dimensions must be equal:$$q = r$$Furthermore, because $w \neq 0$, the projection $P$ is non-zero, meaning $q \ge 1$.

    We now evaluate the element $w$ spatially inside the full algebra $\mathcal{A}$.
Because $w$ is a partial isometry in $\mathcal{A}$, its initial projection $P$ and final projection $Q$ are Murray-von Neumann equivalent in $\mathcal{A}$.

In a finite-dimensional $C^*$-algebra $\mathcal{A} = \mn{\vec{n}}$, two projections are equivalent if and only if they have the exact same matrix rank inside every single direct summand of $\mathcal{A}$.
Thus, for every row $i \in \{1, \dots, l\}$, the spatial rank of $P$ in the $i$-th summand must equal the spatial rank of $Q$ in the $i$-th summand.

The embedding of the $a$-th block into the $i$-th summand is $x \mapsto x \otimes I_{\Lambda_{i,a}}$. Thus, the spatial rank of $P$ in the $i$-th summand is $q \cdot \Lambda_{i,a}$.

The embedding of the $b$-th block into the $i$-th summand is $y \mapsto y \otimes I_{\Lambda_{i,b}}$. Thus, the spatial rank of $Q$ in the $i$-th summand is $r \cdot \Lambda_{i,b}$.

Therefore we get,
$$q \cdot \Lambda_{i,a} = r \cdot \Lambda_{i,b} \quad \text{for all } i \in \{1, \dots, l\}$$
We know $q = r$ and $q \neq 0$. Substituting $q$ for $r$ yields:$$q \cdot \Lambda_{i,a} = q \cdot \Lambda_{i,b}$$
which gives :$$\Lambda_{i,a} = \Lambda_{i,b} \quad \text{for all } i \in \{1, \dots, l\}$$
Therefore, the column vectors are identical:$$\vec{\Lambda}_a = \vec{\Lambda}_b$$This completes the proof.
\end{proof}

\begin{remark}
By Proposition \ref{prop:reg_imply_col_cond}, we can see that $\mn{2} \oplus \mn{2}$ is not regular in $\mn{2}\oplus \mn{4}$ with respect to the inclusion matrix:
\[
\Lambda=
\begin{pmatrix}
1 & 0\\
1 & 1
\end{pmatrix}.
\]
In particular, this is not unitary regular.
\end{remark}

\begin{prop}[Regularity of Inclusions with Identical Columns]\label{prop:reg_identical_cols}
    Let $\vec{d} = (d_1, \dots, d_k)^T$ be a vector of positive integers, and  $\Lambda$ be an $l \times k$ matrix such that all $k$ columns of $\Lambda$ are identical and have strictly positive entries. That is, there exists a vector $\vec{p} = (p_1, \dots, p_l)^T \in \mathbb{N}^l$ with $p_i > 0$ such that:$$\Lambda_{ij} = p_i \quad \text{for all } i \in \{1, \dots, l\} \text{ and } j \in \{1, \dots, k\}$$Let $\vec{n} = (n_1, \dots, n_l)^T$ be the dimension vector given by $\vec{n} = \Lambda \vec{d}$. Then $\mn{\vec{d}}$ is a regular subalgebra of $\mn{\vec{n}}$ with respect to the inclusion matrix $\Lambda$.
\end{prop}

\begin{proof}
    By the hypothesis, the dimension of the $i$-th summand of $\sA$ is:$$n_i = \sum_{j=1}^k \Lambda_{ij} d_j = \sum_{j=1}^k p_i d_j.$$

    This induces a natural block-matrix decomposition of the ambient space $\mn{n_i}$ into a $k \times k$ grid of major blocks, where the $(j, j')$-th major block has dimensions $(d_j p_i) \times (d_{j'} p_i)$.

    Under this decomposition, the local embedding (see the notations at the beginning of this section) of $x = (x_1, \dots, x_k) \in \mn{\vec{d}}$ into the $i$-th summand is purely block-diagonal:$$\Phi_\Lambda^{(i)}(x) = \begin{pmatrix} x_1 \otimes I_{p_i} & 0 & \dots & 0 \\ 0 & x_2 \otimes I_{p_i} & \dots & 0 \\ \vdots & \vdots & \ddots & \vdots \\ 0 & 0 & \dots & x_k \otimes I_{p_i} \end{pmatrix}$$To prove regularity, we must show that the groupoid normalizer spans every $(j, j')$ major block for every summand $i$.

    \noindent{\bf Step 1:} Constructing the Global Normalizer.

    Fix any target block $j$ and source block $j'$ in $\{1, \dots, k\}$ (they can be the same).
    We want to construct global normalizers (in $\grnor{\sA}(\sB)$) that map the $j'$-th block of $\sB:=\sB(\vec{d}, \Lambda)$ to the $j$-th block.

    \begin{itemize}
        \item[1.] Choose a standard matrix unit $E \in \mn{d_j \times d_{j'}}$ (a matrix with a single $1$ at position $(r, s)$ and $0$ elsewhere).
        \item[2.] For each summand $i = 1, \dots, l$, pick an arbitrary unitary matrix $U_i \in \mathcal{U}(\mn{p_i})$.
    \end{itemize}
    Define the global element $W = (W_1, \dots, W_l) \in \mathcal{A}$ such that in the $i$-th summand, $W_i$ has only one non-zero major block, located at position $(j, j')$, given by the Kronecker product $E \otimes U_i$:
    $$W_i = \begin{pmatrix} 0 & \dots & 0 & \dots & 0 \\ \vdots & & \vdots & & \vdots \\ 0 & \dots & E \otimes U_i & \dots & 0 \\ \vdots & & \vdots & & \vdots \\ 0 & \dots & 0 & \dots & 0 \end{pmatrix} \longleftarrow \text{Row block } j$$$$\qquad \qquad \qquad \uparrow \\ \qquad \qquad \text{Col block } j'$$

    \noindent{\bf Step 2:} Verifying the Normalizer Conditions

    By block-matrix multiplication, $W_i^* W_i$ has a single non-zero major block at position $(j', j')$. We compute it:$$(E \otimes U_i)^* (E \otimes U_i) = (E^* \otimes U_i^*) (E \otimes U_i) = (E^* E) \otimes (U_i^* U_i)$$Because $U_i$ is a unitary matrix, $U_i^* U_i = I_{p_i}$.The matrix $E^* E$ is a $d_{j'} \times d_{j'}$ diagonal matrix with a single $1$ at $(s, s)$, which we denote as the projection $P_s$. Thus:$$(W_i^* W_i)_{(j', j')} = P_s \otimes I_{p_i}=\Phi_\Lambda^{(i)}(0, \dots, P_s, \dots, 0)$$
    Since this holds for every $i$, $W^* W \in \mathcal{B}(\vec{d}, \Lambda)$.

    Let $x = (x_1, \dots, x_k) \in \mathcal{B}$. We compute $W_i \Phi_\Lambda^{(i)}(x) W_i^*$.
    The only interaction occurs when the $(j, j')$ block of $W_i$ hits the $(j', j')$ block of $\Phi_\Lambda^{(i)}(x)$. The result lands in the $(j, j)$ major block:
    $$(E \otimes U_i) (x_{j'} \otimes I_{p_i}) (E \otimes U_i)^*  = (E \otimes U_i) (x_{j'} \otimes I_{p_i}) (E^* \otimes U_i^*)  = (E x_{j'} E^*) \otimes (U_i I_{p_i} U_i^*)$$
    Again, because $U_i$ is unitary, $U_i U_i^* = I_{p_i}$. Thus, the resulting block is:$$(E x_{j'} E^*) \otimes I_{p_i}$$
    Because $E$ is a matrix unit picking out the $(r,s)$ scalar of $x_{j'}$, $E x_{j'} E^*$ is exactly $x_{j',rs} P_r$, which is a valid element of $\mn{d_j}$. This result is exactly $\Phi_\Lambda^{(i)}(0, \dots, E x_{j'} E^*, \dots, 0) \in \mathcal{B}$.

    Therefore, $W \in \grnor{\sA}(\sB)$ for any choice of the unitaries $(U_1, \dots, U_l)$.

    Because the unitary matrices $U_1, \dots, U_l$ are chosen completely independently for each summand $i$, we can take linear combinations of these global normalizers to isolate any specific summand $i$ (forcing the normalizer to be $0$ in all $i' \neq i$).

    Once summand $i$ is isolated, notice that our normalizer block in position $(j, j')$ is $E \otimes U_i$.
    \begin{itemize}
        \item[1.] The linear span of the unitary group in $\mn{p_i}$ is the entire matrix algebra $\mn{p_i}$. Thus, varying $U_i$ spans all possibilities for the right side of the tensor product. 
        \item[2.] The linear span of all standard matrix units $E$ is the entire matrix algebra $\mn{d_j \times d_{j'}}$.
    \end{itemize}
    Because $\mn{d_j \times d_{j'}} \otimes \mn{p_i} \cong \mn{(d_j p_i) \times (d_{j'} p_i)}$, the linear combinations of these normalizers generate every possible matrix inside the $(j, j')$ major block of $\mn{n_i}$. Since this holds for every pair of blocks $(j, j')$ and every summand $i$, the groupoid normalizer spans the entirety of $\mathcal{A} = \bigoplus_{i=1}^l \mn{n_i}$.

    This completes the proof.
\end{proof}

\begin{prop}[Unitary Regularity of Inclusions with Identical Columns]\label{prop:unitary_reg_identical_column}
 Let $\vec{d} = (d_1, \dots, d_k)^T$ be a vector of positive integers such that $d_1=\dots =d_k =d$, and  $\Lambda$ be an $l \times k$ matrix such that all $k$ columns of $\Lambda$ are identical and have strictly positive entries. That is, there exists a vector $\vec{p} = (p_1, \dots, p_l)^T \in \mathbb{N}^l$ with $p_i > 0$ such that:$$\Lambda_{ij} = p_i \quad \text{for all } i \in \{1, \dots, l\} \text{ and } j \in \{1, \dots, k\}$$Let $\vec{n} = (n_1, \dots, n_l)^T$ be the dimension vector given by $\vec{n} = \Lambda \vec{d}$. Then $\mn{\vec{d}}$ is a unitary regular subalgebra of $\mn{\vec{n}}$ with respect to the inclusion matrix $\Lambda$.   
\end{prop}
\begin{proof} 
By the hypothesis, the dimension of the $i$-th summand of $\sA$ is:
$$n_i = \sum_{j=1}^k \Lambda_{ij} d_j = \sum_{j=1}^k p_i d = k p_{i}d.$$

  The embedding of $x = (x_1, \dots, x_k)\in \mn{\vec{d}}$  into the $i$-th summand is exactly:
  
  $$\Phi_\Lambda^{(i)}(x) = \begin{pmatrix} x_1 \otimes I_{p_i} & 0 & \dots & 0 \\ 0 & x_2 \otimes I_{p_i} & \dots & 0 \\ \vdots & \vdots & \ddots & \vdots \\ 0 & 0 & \dots & x_k \otimes I_{p_i} \end{pmatrix}$$
where each $x_j \in \mn{d}$. 
    To span the space $\mn{\vec{n}}$, we need a global matrix $W = (W_1, \dots, W_l)$ that is completely unitary everywhere and normalizes $\sB:=\sB(\vec{d}, \Lambda)$. We build it in three pieces:
    
    The Block Shuffle: Pick any permutation $\pi \in S_k$. We will use this to shuffle the $k$ blocks around. Let $E_{\pi(j), j}$ be the $k \times k$ matrix unit that maps block $j$ to block $\pi(j)$.
    
    The Internal Twists: For each block $j$, pick an arbitrary unitary $u_j \in \mathcal{U}(\mn{d})$.
    
    The Multiplicity Mixers: For each summand $i$, pick an arbitrary unitary $U_i \in \mathcal{U}(\mn{p_i})$. We construct the component $W_i$ in the $i$-th summand by combining all three:$$W_i = \sum_{j=1}^k \left( E_{\pi(j), j} \otimes u_j \otimes U_i \right)$$ 
    Because $\pi$ is a permutation, $W_i$ is a block-matrix with exactly one non-zero block in each row and column of blocks. Because $u_j$ and $U_i$ are unitary, each non-zero block is unitary. Therefore, $W_i$ is a strictly unitary matrix ($W_i^* W_i = I_{n_i}$ and $W_i W_i^* = I_{n_i}$).

    Let $x= (x_1,\dots, x_k) \in \mathcal{B}$. Because $W_i$ has exactly one block per row/col, it cleanly shifts the diagonal blocks of $\Phi^{(i)}(x)$ to new positions:$$W_i \Phi^{(i)}(x) W_i^* = \sum_{j=1}^k \left( E_{\pi(j), \pi(j)} \otimes (u_j x_j u_j^*) \otimes (U_i I_{p_i} U_i^*) \right).$$
    
    Since $U_i$ is unitary, $U_i I_{p_i} U_i^* = I_{p_i}$. If we re-index the sum by letting $m = \pi(j)$ (so $j = \pi^{-1}(m)$), the result is:
    $$W_i \Phi^{(i)}(x) W_i^* = \operatorname{diag}\Big( \dots, \; (u_{\pi^{-1}(m)} x_{\pi^{-1}(m)} u_{\pi^{-1}(m)}^*) \otimes I_{p_i}, \; \dots \Big).$$
    
    This is exactly $\Phi^{(i)}(y)$ where the new element $y \in \mathcal{B}$ has blocks given by $y_m = u_{\pi^{-1}(m)} x_{\pi^{-1}(m)} u_{\pi^{-1}(m)}^*$. Note that the new element $y$ depends only on $\pi$ and $u_j$. It does not depend on $i$ or the local unitary $U_i$. Therefore, the output is identical across all summands, meaning $W \mathcal{B} W^* \subseteq \mathcal{B}$. Thus $W$ is a  global unitary normalizer in $\mn{\vec{n}}$.

    Now look at the structure of $W_i$:
    $$W_i = \left( \sum_{j=1}^k E_{\pi(j), j} \otimes u_j \right) \otimes U_i.$$
    
    Let $M = \sum_{j=1}^k E_{\pi(j), j} \otimes u_j$. This $M$ is a $kd \times kd$ "monomial unitary matrix". It is a standard fact of linear algebra that the linear span of all such monomial unitaries (varying $\pi$ and $u_j$) is the entirety of the matrix algebra $\mn{kd}$. So, by taking linear combinations of our normalizers, the $M$ portion spans $\mn{kd}$. Meanwhile, because we can choose $U_i$ completely independently for each summand $i$, and the linear span of $\mathcal{U}(\mn{p_i})$ is $\mn{p_i}$, we can isolate the summands exactly as we did in the proof of Proposition \ref{prop:reg_identical_cols}. The span of the tensor product is the tensor product of the spans:$$\operatorname{span}(W_i) = \mn{kd} \otimes \mn{p_i} \cong \mn{n_i}.$$
    
    Since we span $\mn{n_i}$ for every summand $i$, the unitary normalizers generate the entire ambient algebra $\mn{\vec{n}}$. Therefore the inclusion is unitary regular.
\end{proof}

\begin{theorem}[Characterization of Regular Inclusion]\label{thm:char_regularity}
Let $\vec{n} = (n_1,\dots,n_l)^T\in\mathbb{N}^l$, $\vec{d}=(d_1,\dots,d_k)^T\in\mathbb{N}^k$, and let $\Lambda=[\Lambda_{ij}]_{l\times k}$ satisfy $\Lambda\vec{d}=\vec{n}.$
Then the algebra $\mn{\vec{d}}$ is regular in $\mn{\vec{n}}$ if and only if $\Lambda$ is regular (see definition \ref{def:regular_mat}).
\end{theorem}

\begin{proof}
    \noindent($\implies$).  Assume that $\sB(\vec{d}, \Lambda)$ is regular in $\mn{\vec{n}}$.
    
    Now consider the canonical central projection $p_i: \bigoplus_{i=1}^l\mn{n_i} \to \mn{n_i}$ for every $i\in \{1,\dots, l\}$. We have, for all $i = 1, \dots, l$,
    $$\bigoplus_{j\in S_i}\mn{d_j} \otimes \mathbb{I}_{\Lambda_{ij}}=\bigoplus_{j=1}^k\mn{d_j} \otimes \mathbb{I}_{\Lambda_{ij}}=p_i\left(\sB(\vec{d}, \Lambda)\right)   \subseteq p_i (\mn{\vec{n}})=\mn{n_i}.$$ 
    
    Now by Lemma \ref{cut-down-regularity}, we get that $p_i(\sB(\vec{d}, \Lambda))$ is regular in $\mn{n_i}$. Then by Theorem \ref{thm:regularity-char-mnc} for each $t\in S_i$, $\Lambda_{it}$ are same.

    The necessity of condition (ii) is established in Proposition \ref{prop:reg_imply_col_cond}.

    \vspace{0.5cm}
    \noindent($\impliedby$). Assume $\Lambda$ is regular.

    Define an equivalence relation on the columns (blocks of $\mathcal{B}:=\mn{\vec{d}}$) by $a \sim b$ if and only if $\vec{\Lambda}_a = \vec{\Lambda}_b$. By $(ii)$, if row $i$ has non-zero entries in column $a$ and column $b$, then $a \sim b$. Therefore, we can partition the rows and columns into distinct families. By applying a permutation to the columns (grouping equivalence classes together) and a permutation to the rows (grouping the rows that support each class together), the matrix $\Lambda$ transforms into a strict block-diagonal form:$$\tilde{\Lambda} = \begin{pmatrix}  \Lambda^{(1)} & 0 & \dots & 0 \\  0 & \Lambda^{(2)} & \dots & 0 \\  \vdots & \vdots & \ddots & \vdots \\  0 & 0 & \dots & \Lambda^{(M)}  \end{pmatrix}$$where each block $\Lambda^{(m)}$ is a matrix where all columns are identical and strictly positive. Permuting rows and columns simply corresponds to relabeling the summands of $\mathcal{A}$ and $\mathcal{B}$, which trivially preserves regularity. Thus, we can assume $\Lambda$ is already in this block-diagonal form.

    Because $\Lambda$ is block-diagonal, the abstract algebra $\mathcal{B}$ naturally decomposes into $\mathcal{B}^{(1)} \oplus \dots \oplus \mathcal{B}^{(M)}$ (one for each set of columns). Similarly, the ambient algebra $\mathcal{A}$ decomposes into $\mathcal{A}^{(1)} \oplus \dots \oplus \mathcal{A}^{(M)}$ (one for each set of rows). Crucially, the global embedding $\Phi_\Lambda$ maps $\mathcal{B}^{(m)}$ entirely and exclusively into $\mathcal{A}^{(m)}$. Therefore, the global inclusion splits into a true direct sum of independent inclusions:$$\mathcal{B}(\vec{d}, \Lambda) \subseteq \mathcal{A} \quad \text{is isomorphic to} \quad \bigoplus_{m=1}^M \left( \mathcal{B}^{(m)}(\vec{d}^{(m)}, \Lambda^{(m)}) \subseteq \mathcal{A}^{(m)} \right)$$

    Now by Proposition \ref{prop:reg_identical_cols}, each of $\left( \mathcal{B}^{(m)}(\vec{d}^{(m)}, \Lambda^{(m)}) \subseteq \mathcal{A}^{(m)} \right)$ is regular inclusion and since regularity is preserved under direct sum (by Lemma \ref{lem:reg_pres_direct_sum}), we have $\sB(\vec{d}, \sA)$ is regular.

    This completes the proof!
\end{proof}

As an immediate corollary of Theorem \ref{thm:char_regularity} we get:
\begin{cor}
    A finite-dimensional unital inclusion $\mathcal{B} \subseteq \mathcal{A}$ is regular if and only if, up to a permutation of rows and columns, its inclusion matrix $\Lambda$ is a block-diagonal matrix:$$\Lambda = \begin{pmatrix}  \Lambda^{(1)} & 0 & \dots & 0 \\  0 & \Lambda^{(2)} & \dots & 0 \\  \vdots & \vdots & \ddots & \vdots \\  0 & 0 & \dots & \Lambda^{(M)}  \end{pmatrix}$$where each block $\Lambda^{(m)}$ is a rectangular matrix in which every row consists of strictly positive, equal entries.
\end{cor}
We provide a more streamlined  proof of the characterization of unitary regularity for an inclusion of general multi-matrix algebras (see \cite{Bakshi2026}).
\begin{theorem}[Characterization of Unitary Regular Inclusion]\label{thm:char_unitary_regularity}
Let $\vec{n} = (n_1,\dots,n_l)^T\in\mathbb{N}^l$, $\vec{d}=(d_1,\dots,d_k)^T\in\mathbb{N}^k$, and let $\Lambda=[\Lambda_{ij}]_{l\times k}$ satisfy $\Lambda\vec{d}=\vec{n}.$
Then the algebra $\mn{\vec{d}}$ is unitary regular in $\mn{\vec{n}}$ if and only if the following conditions hold:
\begin{enumerate}
    \item[(i)] $\Lambda$ is regular (see definition \ref{def:regular_mat}), 
    \item[(ii)] for each $ i = 1, \dots, l$ we have $d_j = d_{j'}$ whenever $j, j' \in S_i$, 
\end{enumerate}
where
\[
S_i=\{\,j:\Lambda_{ij}>0\,\}
\]
denotes the support of the $i$th row of $\Lambda$.    
\end{theorem}

\begin{proof}
 \noindent($\implies$).  Assume that $\sB(\vec{d}, \Lambda)$ is unitary regular in $\mn{\vec{n}}$. This implies that $\sB(\vec{d}, \Lambda)$ is regular in $\mn{\vec{n}}$. Thus the condition of Theorem \ref{thm:char_regularity} is satisfied. As in the proof of Theorem \ref{thm:char_regularity}, we now consider the canonical central projection $p_i: \mn{\vec{n}} \to \mn{n_i}$ for every $i\in \{1,\dots, l\}$. We have, for all $i = 1, \dots, l$,
    $$\bigoplus_{j\in S_i}\mn{d_j} \otimes \mathbb{I}_{\Lambda_{ij}}=\bigoplus_{j=1}^k\mn{d_j} \otimes \mathbb{I}_{\Lambda_{ij}}=p_i\left(\sB(\vec{d}, \Lambda)\right)   \subseteq p_i (\mn{\vec{n}})=\mn{n_i}.$$    
    Now by Lemma \ref{cut-down-regularity}, we get that $p_i(\sB(\vec{d}, \Lambda))$ is unitary regular in $\mn{n_i}$. Then by Theorem \ref{thm:uni-reg-char-mnc} for each $j, j^{'} \in S_i$, $\Lambda_{ij}= \Lambda_{ij^{'}}$ and $d_j= d_{j^{'}}$ are same. 

    \vspace{0.5cm}
    \noindent($\impliedby$) As $\Lambda$ is regular, by the same argument as in the converse part of the proof of Theorem \ref{thm:char_regularity}, we get
    $$\mathcal{B}(\vec{d}, \Lambda) \subseteq \mathcal{A} \quad \text{is isomorphic to} \quad \bigoplus_{m=1}^M \left( \mathcal{B}^{(m)}(\vec{d}^{(m)}, \Lambda^{(m)}) \subseteq \mathcal{A}^{(m)} \right).$$

   By condition $(ii)$, we also get that for each $m$, $\vec{d}^{(m)}$ has the same components. Now by Proposition \ref{prop:unitary_reg_identical_column} each of $\left( \mathcal{B}^{(m)}(\vec{d}^{(m)}, \Lambda^{(m)}) \subseteq \mathcal{A}^{(m)} \right)$ is unitary regular inclusion and since unitary regularity is preserved under direct sum (by Lemma \ref{lem:reg_pres_direct_sum}), we have $\sB(\vec{d}, \sA)$ is unitary regular.

    This completes the proof!     

\end{proof}

    \section{Acknowledgement}
    The first named author acknowledges the support of the grant ANRF/ECRG/2024/002328/PMS. Sumit Kumar would like to thank the first named author for the opportunity to work at IIT Kanpur under his guidance through project no. SPO/ANRF/MATH/2025271. We would like to thank Prof. Ved Prakash Gupta for his valuable discussions and insightful suggestions during the early stages of this project.
    
	\medskip
	
	\bibliographystyle{amsalpha} 
	\bibliography{references}

\end{document}